\documentclass[a4paper, 11pt]{article}
\usepackage{geometry}
\usepackage[english]{babel}
\usepackage{amsmath,amssymb,latexsym}
\usepackage{xcolor}
\usepackage{authblk}
\usepackage{color}
\usepackage{physics}
\usepackage{amsthm}
\makeatletter
\def\th@plain{%
  \thm@notefont{}% same as heading font
  \itshape % body font
}
\def\th@definition{%
  \thm@notefont{}% same as heading font
  \normalfont % body font
}
\makeatother
\usepackage{accents}
\usepackage{enumitem}

\usepackage{varwidth}

\usepackage[justification=centering]{caption}
\usepackage[colorlinks=true, citecolor=green, linkcolor=blue]{hyperref}
\usepackage[nameinlink, capitalize]{cleveref}

\makeatletter
\newenvironment{subtheorem}[1]{%
  \def\subtheoremcounter{#1}%
  \refstepcounter{#1}%
  \protected@edef\theparentnumber{\csname the#1\endcsname}%
  \setcounter{parentnumber}{\value{#1}}%
  \setcounter{#1}{0}%
  \expandafter\def\csname the#1\endcsname{\theparentnumber.\alph{#1}}%
  \ignorespaces
}{%
  \setcounter{\subtheoremcounter}{\value{parentnumber}}%
  \ignorespacesafterend
}
\makeatother
\newcounter{parentnumber}

\theoremstyle{plain}
\newtheorem{theorem}{Theorem}

\newtheorem{lemma}{Lemma}
\newtheorem{proposition}{Proposition}

\theoremstyle{definition}

\newtheorem{remark}{Remark}

\newcounter{mnotecount}[section]

\numberwithin{equation}{section}

\begin{document}

\title{The semiclassical limit from the Pauli-Poisswell/Darwin to the Euler-Poisswell/Darwin system by WKB methods}
%\author{Work in progress}
\author[a,b]{Norbert J. Mauser}
\author[a,b,d]{Jakob Möller}
\author[c]{Changhe Yang}
\affil[a]{Research Platform MMM "Mathematics-Magnetism-Materials" c/o Fak. Math., Univ. Wien, Oskar-Morgenstern-Platz 1, 1090 Vienna, Austria}
\affil[b]{Wolfgang Pauli Institut, Vienna, Austria}
\affil[c]{California Institute of Technology, 91125 Pasadena, CA, USA}
\affil[d]{CMLS, École Polytechnique, F-91128 Palaiseau}
\maketitle

%%%%%%%%%%%%%%%%%%%%%%%%%%%%%%%%%%%%%%%%%%%%%%%%%%%%%%%%%%%%%%%%%%%%%%%%%%%%%%%%%%%%%%%%%%%%%%%%%
\begin{abstract}
The self-consistent Pauli-Poisswell and Pauli-Darwin equations for 2-spinors are $O(1/c)$ (where $c$ denotes the speed of light) semi-relativistic approximations of the Dirac-Maxwell equation for 4-spinors coupled to the self-consistent electromagnetic fields generated by the charge and current densities of a fast moving electric charge.
They consist of a vector-valued magnetic Schrödinger equation with the Stern-Gerlach term which couples spin and magnetic field, coupled to 1+3 Poisson equations as the magnetostatic approximation of Maxwell's equations.

The Pauli-Poisswell and Pauli-Dariwn euqations are $O(1/c)$ models keeping both relativistic effects magnetism and spin, both of which are absent in the non-relativistic Schrödinger-Poisson equation and inconsistent in the magnetic Schrödinger-Maxwell equation.

We prove the local in time semiclassical limit $\hbar \rightarrow 0$ to the Euler-Poisswell equation and Euler-Darwin equations based on WKB analysis
and energy estimates. Moreover we obtain weak convergence of the monokinetic Wigner transform to the monokinetic scalar Wigner measure solving the Vlasov-Poisswell and Vlasov-Darwin equations and strong convergence of the macroscopic densities. We introduce the Euler-Poisswell/Darwin equation and prove local wellposedness and a blow up alternative. \\

\noindent \textbf{Key words:} Pauli equation; Vlasov-Darwin equation; Vlasov-Poisswell equation; Euler-Darwin equation; Euler-Poisswell equation; WKB approximation; Semiclassical limit;  \\
\textbf{MSC: 81Q20; 35Q40; 35Q35}
\end{abstract}
%%%%%%%%%%%%%%%%%%%%%%%%%%%%%%%%%%%%%%%%%%%%%%%%%%%%%%%%%%%%%%%%%%%%%%%%%%%%%%%%%%%%%%%%%%%%%%%%%

\setcounter{tocdepth}{1}
\tableofcontents

%%%%%%%%%%%%%%%%%%%%%%%%%%%%%%%%%%%%%%%%%%%%%%%%%%%%%%%%%%%%%%%%%%%%%%%%%%%%%%%%%%%%%%%%%%%%%%%%%
%%%%%%%%%%%%%%%%%%%%%%%%%%%%%%%%%%%%%%%%%%%%%%%%%%%%%%%%%%%%%%%%%%%%%%%%%%%%%%%%%%%%%%%%%%%%%%%%%
\section{Introduction}

We study the semiclassical limit of vanishing Planck constant $\hbar$ of two nonlinear systems of PDE consisting of the Pauli equation, introduced by Wolfgang Pauli in \cite{pauli1927quantenmechanik}, for the unknown 2-spinor $\psi^{\hbar,c}$ with values in $\mathbb{C}^2$, depending on two parameters $\hbar$ (Planck constant) and $c$ (speed of light), coupled to semi-relativistic approximations of the Maxwell equations for the unknown scalar electric potential $V^{\hbar,c}$ (taking values in $\mathbb{R}$) and the vector-valued magnetic potential $A^{\hbar,c}$ (taking values in $ \mathbb{R}^3$), where the source terms are the scalar charge density $\rho^{\hbar,c}:=|u^{\hbar,c}|^2$ and the vector-valued current density $J^{\hbar,c}$.

The \textbf{Pauli-Poisswell equation} in Lorenz\footnote{named after Ludvig Lorenz (1829-1891) and often falsely attributed to Hendrik Antoon Lorentz (1853-1928), after whom the Lorentz transformation and the Lorentz force are named.} gauge \eqref{eq:lorenz gauge} is given by
\begin{align}
    i\hbar\partial_t \psi^{\hbar,c} &= -\frac{1}{2}(\hbar \nabla-\frac{i}{c}A^{\hbar,c})^2\psi^{\hbar,c} + V^{\hbar,c} \psi^{\hbar,c} -\frac{1}{2} \frac{\hbar}{c} (\sigma \cdot B^{\hbar,c}) \psi^{\hbar,c}, \label{eq:PPW_Pauli}\\
    -\Delta V^{\hbar,c} &= \rho^{\hbar,c} := | \psi^{\hbar,c}|^2, \label{eq:PPW_V} \\
    -\Delta A^{\hbar,c} &= \frac{1}{c}J^{\hbar,c}, \quad B^{\hbar,c}:=\nabla \times A^{\hbar,c},\label{eq:PPW_PoissonA} \\
    \text{div}A^{\hbar,c} + \frac{1}{c}\partial_t V^{\hbar,c} &= 0, \label{eq:lorenz gauge}\\
    \psi^{\hbar,c}(x,0) &= \psi^{\hbar,c,0}(x).
    \label{eq:PPW_data}
\end{align}
The \textbf{Pauli-Darwin equation} in Coulomb gauge \eqref{eq:coulomb gauge} is given by
\begin{align}
    i\hbar\partial_t \psi^{\hbar,c} &= -\frac{1}{2}(\hbar \nabla-\frac{i}{c}A^{\hbar,c})^2\psi^{\hbar,c} + V^{\hbar,c} \psi^{\hbar,c} -\frac{1}{2} \frac{\hbar}{c} (\sigma \cdot B^{\hbar,c}) \psi^{\hbar,c}, \label{eq:PD_Pauli}\\
    -\Delta V^{\hbar,c} &= \rho^{\hbar,c} := | \psi^{\hbar,c}|^2, \label{eq:PD_V}\\
    -\Delta A^{\hbar,c} &= \frac{1}{c}J^{\hbar,c}-\frac{1}{c}\partial_t\nabla V^{\hbar,c}, \quad B^{\hbar,c}:=\nabla \times A^{\hbar,c},\label{eq:PD_PoissonA} \\
    \text{div}A^{\hbar,c}  &= 0, \label{eq:coulomb gauge}\\
    \psi^{\hbar,c}(x,0) &= \psi^{\hbar,c,0}(x).
    \label{eq:PD_data}
\end{align}
For both equations the Pauli current density $J^{\hbar,c}$ is given by
\begin{equation}
    J^{\hbar,c} = \Im\langle {\psi^{\hbar,c}},(\hbar\nabla -\frac{i}{c}A^{\hbar,c})\psi^{\hbar,c}\rangle +\frac{\hbar}{2}\nabla \times \langle{\psi^{\hbar,c}} ,\sigma \psi^{\hbar,c}\rangle. \label{eq:PPW_current}
\end{equation}
Here we used the following notations
\begin{itemize}
\item $ \mathbf{\sigma} \cdot B^{\hbar,c} = \sum_{k=1}^3 \sigma_k B^{\hbar,c}_k$ denotes the Stern-Gerlach term, where the $\{\sigma_i\}$ are the Pauli matrices (cf. equation \eqref{eq:pauli_matrices} in Section \ref{sec:notations}).
\item $\langle \cdot , \cdot \rangle = \langle \cdot , \cdot \rangle_{\mathbb{C}^2}$ denotes the inner product in $\mathbb{C}^2$ (antilinear in the first variable).
\item $\rho^{\hbar,c} := |\psi^{\hbar,c}|^2:= \langle\psi^{\hbar,c},\psi^{\hbar,c}\rangle = |\psi^{\hbar,c}_1|^2+|\psi^{\hbar,c}_2|^2$ is the scalar charge density
\item $\langle \psi^{\hbar,c}, \nabla \psi^{\hbar,c} \rangle$ is the 3-vector with components $\langle \psi^{\hbar,c},\partial_k \psi^{\hbar,c}\rangle$, $k=1,2,3$. 
\item Similarly, $\langle \psi^{\hbar,c}, A^{\hbar,c} \psi^{\hbar,c} \rangle$ and $\langle \psi^{\hbar,c}, \sigma \psi^{\hbar,c} \rangle$ are the 3-vectors with entries $\langle \psi^{\hbar,c},A^{\hbar,c}_k \psi^{\hbar,c}\rangle$ and $\langle \psi^{\hbar,c}, \sigma_k \psi^{\hbar,c}\rangle$.
\end{itemize}
The self-consistent Pauli-Poisswell (a portmanteau of \emph{Poiss}on and Max\emph{well}, coined in \cite{masmoudi2001selfconsistent}) and Pauli-Darwin equations arise as $O(1/c)$ semi-relativistic approximations of the relativistic Dirac-Maxwell equation for a 4-spinor \cite{GeMM23, masmoudi2001selfconsistent}, where the Poisson equations for $A^{\hbar,c}$ and $V^{\hbar,c}$ are the magnetostatic approximation of Maxwell's equation by keeping terms of order $O(1/c)$ and $O(1/c^2)$, respectively, cf. \cite{besse2007,krause2007unified,shiroto2023improved}. They are semi-relativistic extensions of the Schrödinger-Poisson equation in $\mathbb{R}^3$ via the magnetic Laplacian and the Stern-Gerlach term $\sigma \cdot B^{\hbar,c}$, therefore keeping magnetic and spin effects. 

The two components of the Pauli equation describe the two spin states of a charged spin-$1/2$-particle (i.e. a fermion), whereas the Poisson equations describe the electrodynamic self-interaction of a fast moving particle.
In $\mathbb{R}^3$, the Poisson equation $-\Delta V^{\hbar,c} = \rho^{\hbar,c}$ arises from the Hartree nonlinearity for the Coulomb potential $|x|^{-1}$,
\begin{equation*}
  V^{\hbar,c} = \frac{1}{4 \pi |x|} \ast \rho^{\hbar,c} = (- \Delta)^{- 1} 
|\psi^{\hbar,c}|^2 \qquad \text{in } \mathbb{R}^3 \text{ only.}
\end{equation*}
Since solutions to the Pauli-Poisswell/Darwin equation satisfy the continuity equation
\begin{equation}
  \partial_t \rho^{\hbar,c} + \mathrm{div} J^{\hbar,c} = 0, \label{rho J}
\end{equation}
the "Darwin current" $J^{\hbar,c}-\partial_t \nabla V^{\hbar,c}$ can be written as
\begin{equation}
    J^{\hbar,c}-\partial_t \nabla V^{\hbar,c} = \mathbb{P}J^{\hbar,c},
    \label{eq:leray}
\end{equation}
where $\mathbb{P}$ is the Leray projection on divergence-free vector fields, defined by the Fourier multiplier with matrix-valued symbol
\begin{equation*}
    m_{ij}(\xi) = \delta_{ij} - \frac{\xi_i \xi_j }{|\xi|^2}, \quad 1\leq i
    ,j\leq 3,
\end{equation*}
which is a bounded operator $\mathbb{P}\colon L^p \rightarrow L^p$ for all $1<p<\infty$. \\

We prove the semiclassical limit of vanishing Planck constant $\hbar$ of the Pauli-Poisswell equation \eqref{eq:PPW_Pauli}-\eqref{eq:PPW_data} and the Pauli-Darwin equation \eqref{eq:PD_Pauli}-\eqref{eq:PD_data}.  There are 2 main techniques for such semiclassical limits:

a) \emph{Wigner transforms} \cite{GMMP} yield global (in time) limits towards Vlasov type equations (cf. \cite{LionsPaul, MarkowichMauser}  for the limit of the Schrödinger-Poisson equation to the Vlasov-Poisson equation for \emph{mixed states} in $\mathbb{R}^3$; cf. \cite{ZhZM02} for \emph{pure states} in one space dimension for appropriate non-unique measures valued solutions of the Vlasov-Poisson equation). We introduce the Wigner transform and Wigner measures in section \ref{sec:Wigner}.

b) \emph{WKB methods} yield local (in time) limits to Euler type equations which correspond to monokinetic Vlasov equations for pure states. For the WKB analysis of the Schrödinger-Poisson equation and its convergence towards the Euler-Poisson equation see \cite{carles2007,Carl07,li2003,masaki2011,2003Wigner,Zhan08}. The limit of the cubic NLS was treated in \cite{10.2307/118716,liu2002} and the limit of the Gross-Pitaevskii equation in \cite{gui2022}.

A unified presentation of these two approaches and extension to multi-valued WKB can be found in \cite{carles2012, sparber2003wigner}. \\

In this paper we deal with case b). We prove that in the limit $\hbar \rightarrow 0$ the Pauli-Poisswell equation \eqref{eq:PPW_Pauli}-\eqref{eq:PPW_data}
converges to the \textbf{Euler-Poisswell equation} for the unkown $(\rho^c,u^c)\in \mathbb{R}\times \mathbb{R}^3$ with the electomagnetic fields $E^c,B^c \in \mathbb{R}^3$, depending on the parameter $c$ (speed of light),
\begin{align}
  \label{Euler-Poisswell-Lorentz fields}
   \partial_t \rho^c + \nabla \cdot (\rho^c u^c) &=0,\\
    \partial_t u^c + u^c \cdot \nabla u^c  &= E^c+ \frac{1}{c}u^c\times B^c, \\
     E^c &= -\nabla V^c - \frac{1}{c}\partial_t A^c,  \quad
   B^c 
  =  \nabla \times A^c, \\
   -\Delta V^c &= \rho^c\qquad -\Delta A^c = \frac{1}{c}\rho^c u^c
  \label{E,B Euler Poisswell-Lorentz fields}\\
   (\rho^c,u^c)(x,0) &= (\rho^{c,0},u^{c,0}) (x)\label{Euler-Poisswell-Lorentz fields data}.
\end{align}
and that the Pauli-Darwin equation \eqref{eq:PD_Pauli}-\eqref{eq:PD_data}
converges to the \textbf{Euler-Darwin equation},
\begin{align}
  \label{Euler-Darwin-Lorentz fields}
   \partial_t \rho^c + \nabla \cdot (\rho^c u^c) &=0,\\
    \partial_t u^c + u^c \cdot \nabla u^c  &= E^c+ \frac{1}{c}u^c\times B^c, \\
     E^c &= -\nabla V^c - \frac{1}{c}\partial_t A^c,  \quad
   B^c 
  =  \nabla \times A^c, \\
   -\Delta V^c &= \rho^c\qquad -\Delta A^c = \frac{1}{c}\mathbb{P}(\rho^c u^c)
  \label{E,B Euler Darwin-Lorentz fields}\\
   (\rho^c,u^c)(x,0) &= (\rho^{c,0},u^{c,0}) (x)\label{Euler-Darwin-Lorentz fields data}.
\end{align}
The Euler-Poisswell and Euler-Darwin equations are the $O(1/c)$ semi-relativistic approximations of the Euler-Maxwell equation, which describes the fully relativistic self-consistent interaction with the electromagnetic field, cf. \cite{moller2023euler}.

We will also show that the associated monokinetic Wigner measure $f$ solves the \textbf{Vlasov-Poisswell equation} \begin{gather}
\label{eq:vlasov_limit}
     \partial_t f + p \cdot \nabla_x f +(E + p\times B)\cdot \nabla_p f = 0, \quad E =-\nabla_x V-\partial_t A, \quad B = \nabla \times A, \\
    -\Delta V(x,t) = \rho(x,t) := \int f(x,p,t) \dd p, \label{eq:Poisson_limitV} \\
    -\Delta A(x,t) = J(x,t) := \int pf(x,p,t) \dd p, \label{eq:Poisson_limitA} \\
    f(x,p,0) = f^0(x,p).
    \label{eq:data_limit}
\end{gather}
and the \textbf{Vlasov-Darwin equation} \eqref{eq:vlasov_darwin_limit}-\eqref{eq:data_limit Darwin}
\begin{gather}
\label{eq:vlasov_darwin_limit}
     \partial_t f + p \cdot \nabla_x f +(E + p\times B)\cdot \nabla_p f = 0, \quad E =-\nabla_x V-\partial_t A, \quad B = \nabla \times A, \\
    -\Delta V(x,t) = \rho(x,t) := \int f(x,p,t) \dd p, \label{eq:Poisson_limitV Darwin} \\
    -\Delta A(x,t) = \mathbb{P}J(x,t) := \mathbb{P}\int pf(x,p,t) \dd p, \label{eq:Poisson_limitA Darwin} \\
    f(x,p,0) = f^0(x,p).
    \label{eq:data_limit Darwin}
    \end{gather}
    The Vlasov-Poisswell and Vlasov-Darwin equations are $O(1/c)$ semi-relativistic approximations of the Vlasov-Maxwell equation \cite{bauer2005, besse2007,hankwan2018, pallard2006, seehafer2008}.
    
 We discuss the asymptotic hierarchy of the Euler/Vlasov-Poisswell and Euler/Vlasov-Darwin equations within relativistic fluid/kinetic models with electromagnetic self-interaction in section \ref{physical motivation}.

The analysis of the Pauli-Poisswell and Pauli-Darwin equations is much harder than for the (magnetic) Schrödinger-Poisson equation: Due to the existence of zero modes \cite{erdHos1997semiclassical}, even the case of constant magnetic fields shows hard technical challenges. Moreover, one faces the loss of one derivative in the advective term $A\cdot \nabla$ and the nonlinear coupling of the magnetic vector potential $A$ via the current density $J$. 

Local wellposedness and global weak solutions of the Pauli-Poisswell and Pauli-Darwin equations are treated in \cite{GeMM23}. The related {Pauli-Poisson equation} was solved in \cite{moller2023models}, where an external magnetic field is applied %which is much stronger than the magnetic self-interaction
while the electric field $V$ is kept self-consistent. The semiclassical limit of the Pauli-Poisson equation to the magnetic Vlasov-Poisson equation with Lorentz force was proved in \cite{moller2023poisson} as an extension of the results for the Schrödinger-Poisson equation in \cite{LionsPaul, MarkowichMauser}. \\

The precise statements about the wellposedness and semiclassical limit can be found in Theorems \ref{thm:main1} and \ref{thm:main result semiclassical limit} in section \ref{main results}.

%%%%%%%%%%%%%%%%%%%%%%%%%%%%%%%%%%%%%%%%%%%%%%%%%%%%%%%%%%%%%%%%%%%%%%%%%%%
%%%%%%%%%%%%%%%%%%%%%%%%%%%%%%%%%%%%%%%%%%%%%%%%%%%%%%%%%%%%%%%%%%%%%%%%%%%
\section{Statement of the main result}\label{2}

We use a scaling  where the dimensionless semiclassical parameter $\varepsilon$ is proportional to $\hbar$ while keeping $c$ fixed and will omit the $c$-supercript in the remainder of the paper. Thus, the Pauli-Poisswell equation \eqref{eq:PPW_Pauli}-\eqref{eq:PPW_data} reads
\begin{align}
    i\varepsilon\partial_t \psi^{\varepsilon} &= -\frac{1}{2}(\varepsilon \nabla-iA^{\varepsilon})^2\psi^{\varepsilon} + V^{\varepsilon} \psi^{\varepsilon} -\frac{1}{2} \varepsilon (\sigma \cdot B^{\varepsilon}) \psi^{\varepsilon}, \label{eq:PPW_Pauli_scaled}\\
    -\Delta V^{\varepsilon} &= \rho^{\varepsilon} := |\psi^{\varepsilon}|^2, \label{eq:PPW_PoissonV_scaled} \\
    -\Delta A^{\varepsilon}&= J^{\varepsilon}, \label{eq:PPW_PoissonA_scaled} \\ \text{div}A^{\varepsilon} + \partial_t V^{\varepsilon} &= 0, \\
    \psi^{\varepsilon}(x,0) &= \psi^{\varepsilon,0}(x).
\end{align}
and the Pauli-Darwin equation \eqref{eq:PD_Pauli}-\eqref{eq:PD_data} reads, using the Leray projection \eqref{eq:leray},
\begin{align}
    i\varepsilon\partial_t \psi^{\varepsilon} &= -\frac{1}{2}(\varepsilon \nabla-iA^{\varepsilon})^2\psi^{\varepsilon} + V^{\varepsilon} \psi^{\varepsilon} -\frac{1}{2} \varepsilon (\sigma \cdot B^{\varepsilon}) \psi^{\varepsilon}, \label{eq:PD_Pauli_scaled}\\
    -\Delta V^{\varepsilon} &= \rho^{\varepsilon} := |\psi^{\varepsilon}|^2, \label{eq:PD_PoissonV_scaled} \\
    -\Delta A^{\varepsilon}&= \mathbb{P}J^{\varepsilon}, \label{eq:PD_PoissonA_scaled} \\ \text{div}A^{\varepsilon} &= 0, \\
    \psi^{\varepsilon}(x,0) &= \psi^{\varepsilon,0}(x).
\end{align}
with
\begin{equation}
    J^{\varepsilon}= \Im\langle{\psi^{\varepsilon}},(\varepsilon\nabla -iA^{\varepsilon})\psi^{\varepsilon}\rangle + \frac{\varepsilon}{2}\nabla \times \langle{\psi^{\varepsilon}} ,\sigma \psi^{\varepsilon}\rangle.\label{eq:PPW_current_scaled}
\end{equation}

%

%%%%%%%%%%%%%%%%%%%%%%%%%%%%%%%%%%%%%%%%%%%%%%%%%%%%%%%%%%%%%%%%%%%%%%%%%%%
\subsection{WKB Ansatz and formal limit}

We assume that the initial data $\psi^{\varepsilon,0} = (\psi^{\varepsilon,0}_1,\psi^{\varepsilon,0}_2)^T \in \mathbb{C}^2$  is of the form
\begin{equation} \label{WKB initial} 
  \begin{pmatrix}
      \psi^{\varepsilon,0}_1 (x) \\
      \psi^{\varepsilon,0}_2 (x)
  \end{pmatrix} = \begin{pmatrix}
      a^{\varepsilon,0}_1 (x) \\
      a^{\varepsilon,0}_2 (x)
  \end{pmatrix}  e^{\frac{i}{\varepsilon} S^{\varepsilon, 0}
  (x)}
\end{equation}
where  $a^{\varepsilon,0}_j$ are the \emph{initial amplitudes} of the components of the 2-spinor. We choose the real-valued \emph{initial phase} $S^{\varepsilon,0}$ to be the same for both spinor components. One then expects that at least for short times the solution $\psi^{\varepsilon}$ will be of the same form
\begin{equation} \label{WKB} 
  \begin{pmatrix}
      \psi^{\varepsilon}_1 (x,t) \\
      \psi^{\varepsilon}_2 (x,t)
  \end{pmatrix} = \begin{pmatrix}
      a^{\varepsilon}_1 (x,t) \\
      a^{\varepsilon}_2 (x,t)
  \end{pmatrix}  e^{\frac{i}{\varepsilon} S^{\varepsilon, 0}
  (x,t)}
\end{equation}
where $a^{\varepsilon}_j$ are the \emph{amplitudes}, $S^{\varepsilon}$ is the \emph{phase}. Writing \eqref{WKB} in spinor notation gives  \begin{equation}
  \label{WKB2} \psi^{\varepsilon}  (x, t) = a^{\varepsilon} (x, t) e^{\frac{i}{\varepsilon} S^{\varepsilon} (x, t)} .
\end{equation}  
where $\psi^{\varepsilon} = (\psi_1^{\varepsilon},\psi_2^{\varepsilon})^T \in \mathbb{C}^2$, $a^{\varepsilon} = (a_1^{\varepsilon},a_2^{\varepsilon})^T\in \mathbb{C}^2$ and $S^{\varepsilon} \in \mathbb{R}$.

\begin{remark} We choose the same phase $S^{\varepsilon,0}$ for the two components of  $\psi^{\varepsilon,0}$ in \eqref{WKB initial}. When dealing with matrix-valued Hamiltonians this  \emph{multicomponent WKB ansatz} (vector-valued amplitude, scalar phase) is a sound mathematical choice and usually chosen in the physics literature \cite{2003Keppeler}, \cite{1966Yamasaki}. 
Choosing two individual phases $S^{\varepsilon,0}_j$, $j=1,2$ would yield  oscillatory cross terms like $$
      \exp(\tfrac{i}{\varepsilon}(S^{\varepsilon}_1-S^{\varepsilon}_2)),
 $$which would complicate the mathematical situation considerably. 
\end{remark}

%%%%%%%%%%%%%%%%%%%%%%%%%%%%%%%%%%%%%%%%%%%%%%%%%%
%We use the abbreviation $\mathbf{B}^{\varepsilon} = \sigma \cdot B^{\varepsilon} = \sum_{k=1}^3 \sigma_k B^{\varepsilon}_k$.  %and the short-hand notation
%\begin{equation}    (\mathbf{B}^{\varepsilon} a^{\varepsilon})_i := ((\sigma \cdot B^{\varepsilon}) a^{\varepsilon})_i := \sum_{j=1}^2\left(\sum_{k=1}^3 (\sigma_k B^{\varepsilon}_k)_{ij} \right)a^{\varepsilon}_j, \quad i=1,2
%\end{equation}
Substituting the ansatz \eqref{WKB} into \eqref{eq:PPW_Pauli_scaled} yields
\begin{equation*}
    \begin{split}
     &i \varepsilon (\partial_t a^{\varepsilon} + \frac{i}{\varepsilon} a^{\varepsilon} \partial_t S^{\varepsilon})
   e^{\frac{i}{\varepsilon} S^{\varepsilon}} \\ &\qquad = - \frac{\varepsilon^2}{2} \left(\Delta a^{\varepsilon} +
   \frac{2 i}{\varepsilon} \nabla a^{\varepsilon} \cdot \nabla S^{\varepsilon} + \frac{i}{\varepsilon}
   a^{\varepsilon} \Delta S^{\varepsilon} - \frac{1}{\varepsilon^2} | \nabla S^{\varepsilon}|^2 a^{\varepsilon}\right)
   e^{\frac{i}{\varepsilon} S^{\varepsilon}}  \\
   &\qquad \qquad + \left(i \varepsilon A^{\varepsilon} \cdot (\nabla a^{\varepsilon} +
   \frac{i}{\varepsilon} a^{\varepsilon} \nabla S^{\varepsilon})  + (\frac{i\varepsilon}{2}\text{div}A^{\varepsilon} +
   \frac{|A^{\varepsilon}|^2}{2} + V^{\varepsilon} -
   \frac{\varepsilon}{2}  (\sigma \cdot B^{\varepsilon})) a^{\varepsilon}\right)e^{\frac{i}{\varepsilon} S^{\varepsilon}} .
   \end{split}
\end{equation*}
Regroup the terms in the equation above such that
\begin{equation*}
    \begin{split}
  &a^{\varepsilon} \left( \partial_t S^{\varepsilon} + \frac{1}{2} | \nabla S^{\varepsilon}|^2 - A^{\varepsilon} \cdot \nabla S^{\varepsilon} + (\frac{|A^{\varepsilon}|^2}{2} + V^{\varepsilon})
   \right) \\
  - &i \varepsilon \left( \partial_t a^{\varepsilon} +(\nabla S^{\varepsilon} - A^{\varepsilon}) \cdot \nabla a^{\varepsilon} +
    \frac{1}{2} a^{\varepsilon}  (\Delta S^{\varepsilon} - \mathrm{div} A^{\varepsilon}) - \frac{i\varepsilon}{2} \Delta a^{\varepsilon}-
  \frac{i}{2}  (\sigma \cdot B^{\varepsilon}) a^{\varepsilon}\right) = 0.
  \label{extension}
  \end{split}
\end{equation*}
Then we obtain the following system of the \emph{transport equation} for $a^{\varepsilon}$ and the \emph{eikonal equation} for $S^{\varepsilon}$,
\begin{align}
    %\begin{split}
  \label{Madelung-Poisson 0.5}
    \partial_t a^{\varepsilon}+ (\nabla S^{\varepsilon} - A^{\varepsilon}) \cdot \nabla a^{\varepsilon} +
    \frac{1}{2} a^{\varepsilon}  (\Delta S^{\varepsilon} - \mathrm{div} A^{\varepsilon}) &= \frac{i \varepsilon}{2}
    \Delta a^{\varepsilon} + \frac{i}{2}   (\sigma \cdot B^{\varepsilon}) a^{\varepsilon}, \\
    \partial_t S^{\varepsilon} + \frac{1}{2} | \nabla S^{\varepsilon}|^2 - A^{\varepsilon} \cdot \nabla S^{\varepsilon} +
    (\frac{|A^{\varepsilon}|^2}{2} + V^{\varepsilon}) &= 0. \label{Madelung-Poisson 0.5 2}% \\(a^{\varepsilon},S^{\varepsilon}) (x,0) &= (a^{\varepsilon,0}, S^{\varepsilon,0}) (x)
  %\end{split} 
\end{align}
Define the \emph{velocity} and \emph{initial velocity} as $u^{\varepsilon,0}:= \nabla S^{\varepsilon,0}$ and $u^{\varepsilon} = \nabla S^{\varepsilon}$. After applying $\nabla$ to \eqref{Madelung-Poisson 0.5 2} we obtain the \textbf{Pauli-Poisswell-WKB equation}
\begin{align}
  \label{Madelung-Poisson}
    \partial_t a^{\varepsilon} + (u^{\varepsilon}-A^{\varepsilon}) \cdot \nabla a^{\varepsilon} + \frac{1}{2} a^{\varepsilon} \mathrm{div} (u^{\varepsilon}-A^{\varepsilon})
     &= \frac{i \varepsilon}{2} \Delta a^{\varepsilon} + \frac{i}{2}  (\sigma \cdot B^{\varepsilon}) a^{\varepsilon},\\ \label{Madelung-Poisson u}
    \partial_t u^{\varepsilon} + (u^{\varepsilon}-A^{\varepsilon}) \cdot \nabla u^{\varepsilon} + u^{\varepsilon} \cdot \nabla  A^{\varepsilon} + \nabla
    (\frac{|A^{\varepsilon}|^2}{2} + V^{\varepsilon}) &= 0,\\
      - \Delta V^{\varepsilon} &=  \rho^{\varepsilon} := |a^{\varepsilon}|^2,  \label{V in wkb}\\
  - \Delta A^{\varepsilon} &=   J^{\varepsilon}, 
  \label{A in wkb} \\
    (a^{\varepsilon},u^{\varepsilon}) (x,0) &= (a^{\varepsilon,0}, u^{\varepsilon,0}) (x) \label{wkb data}.
\end{align}
and the \textbf{Pauli-Darwin-WKB equation}
\begin{align}
  \label{Madelung-Poisson Darwin}
    \partial_t a^{\varepsilon} + (u^{\varepsilon}-A^{\varepsilon}) \cdot \nabla a^{\varepsilon} + \frac{1}{2} a^{\varepsilon} \mathrm{div} u^{\varepsilon}
     &= \frac{i \varepsilon}{2} \Delta a^{\varepsilon} + \frac{i}{2}  (\sigma \cdot B^{\varepsilon}) a^{\varepsilon},\\ \label{Madelung-Poisson u Darwin}
    \partial_t u^{\varepsilon} + (u^{\varepsilon}-A^{\varepsilon}) \cdot \nabla u^{\varepsilon} + u^{\varepsilon} \cdot \nabla  A^{\varepsilon} + \nabla
    (\frac{|A^{\varepsilon}|^2}{2} + V^{\varepsilon}) &= 0,\\
      - \Delta V^{\varepsilon} &=  \rho^{\varepsilon} := |a^{\varepsilon}|^2,  \label{V in wkb Darwin}\\
  - \Delta A^{\varepsilon} &=   \mathbb{P}J^{\varepsilon}, 
  \label{A in wkb Darwin} \\
    (a^{\varepsilon},u^{\varepsilon}) (x,0) &= (a^{\varepsilon,0}, u^{\varepsilon,0}) (x). \label{wkb data darwin}
\end{align}
The current density $J^{\varepsilon}$ is given by
\begin{equation}
  J^{\varepsilon} = \varepsilon w^{\varepsilon} + \varepsilon v^{\varepsilon} + \rho^{\varepsilon} (u^{\varepsilon} - A^{\varepsilon}), \label{J exp}
\end{equation}
where
\begin{align}
  v^{\varepsilon} := \frac{1}{2} \nabla \times \langle {a^{\varepsilon}} ,\sigma  a^{\varepsilon}\rangle  %\label{v def}\\
  &&
  w^{\varepsilon} :=  \Im\langle{a^{\varepsilon}} ,\nabla a^{\varepsilon} \rangle 
  \label{w def}
\end{align}
%There is a one-to-one match between the solution $\psi^{\varepsilon}$ of the Pauli-Poisswell equation \eqref{eq:PPW_Pauli_scaled} (upto a constant $e^{i \theta}$) and the solution $(a^{\varepsilon}, u^{\varepsilon})$ of(\ref{Madelung-Poisson}). 
%Given $(a^{\varepsilon}, u^{\varepsilon})$, $\psi^{\varepsilon}$ is retrieved via(\ref{WKB}). On the other hand, given $\psi^{\varepsilon}$ solving \eqref{eq:PPW_Pauli_scaled} one calculates $V^{\varepsilon}$ and $A^{\varepsilon}$ and obtains $(a^{\varepsilon},u^{\varepsilon})$ via\eqref{Madelung-Poisson}-\eqref{Madelung-Poisson u}. Here one needs existence and uniquenessof solutions to \eqref{Madelung-Poisson}-\eqref{wkb data}. 
%\begin{remark}
Since the characteristics of \eqref{Madelung-Poisson}-\eqref{wkb data}
%\begin{equation*}
%\label{characteristics} \frac{d}{dt}{X}(t) = u^{\varepsilon}(X(t),t) - A^{\varepsilon}(X(t),t), \quad X (t_0) = x_0 .
%\end{equation*}
can intersect in finite time (i.e. caustics appear \cite{sparber2003wigner}), this approach is only local in time.
%\end{remark}

\begin{remark}
    In the WKB ansatz one chooses a complex-valued amplitude $a^{\varepsilon}$ in contrast to the Madelung transform, where one uses a real-valued amplitude. This additional degree of freedom avoids the singular quantum pressure in the equation for the velocity $u^{\varepsilon}$. Instead we obtain a skew-symmetric term $i\varepsilon \Delta a^{\varepsilon}$ in the transport equation for $a^{\varepsilon}$ \cite{10.2307/118716}. More details on the Madelung transform and its connection to the hydrodynamic formulation of nonlinear Schrödinger equations can be found in \cite{carles2012}.
\end{remark}
We can formally pass to the limit $\varepsilon \rightarrow 0$ in
\eqref{Madelung-Poisson}-\eqref{Madelung-Poisson u} and \eqref{Madelung-Poisson Darwin}-\eqref{wkb data darwin} and obtain the Euler-Poisswell equation
\begin{align}
  \label{Euler-Poisson}
    \partial_t a + (u-A) \cdot \nabla a + \frac{1}{2} a \mathrm{div} (u-A) &= \frac{i}{2}  (\sigma \cdot B)a,\\
    \partial_t u + (u-A) \cdot \nabla u + u \cdot \nabla  A + \nabla
    (\frac{|A|^2}{2} + V) &= 0\\
    -\Delta V &= \rho = |a|^2,  \label{V Euler Poisswell limit}\\
  -\Delta A &=   \rho (u - A), 
  \label{A Euler Poisswell limit} \\
    (a,u) (x,0) &= (a^0,u^0) (x),\label{Euler Poisswell limit data}
\end{align}
and the Euler-Darwin equation
\begin{align}
  \label{Euler-Darwin}
    \partial_t a + (u-A) \cdot \nabla a + \frac{1}{2} a \mathrm{div} u &= \frac{i}{2}  (\sigma \cdot B)a,\\
    \partial_t u + (u-A) \cdot \nabla u + u \cdot \nabla  A + \nabla
    (\frac{|A|^2}{2} + V) &= 0\\
    -\Delta V &= \rho = |a|^2,  \label{V Euler Darwin limit}\\
  -\Delta A &=   \mathbb{P}(\rho (u - A)), 
  \label{A Euler Darwin limit} \\
    (a,u) (x,0) &= (a^0,u^0) (x),\label{Euler Darwin limit data}
\end{align}
Taking the $\mathbb{C}^2$ inner product with $a$ in \eqref{Euler-Poisson} and taking the real part together with the transformation $u-A\rightarrow u$ yields the Euler-Poisswell equation as in \eqref{Euler-Poisswell-Lorentz fields}-\eqref{Euler-Poisswell-Lorentz fields data} and the Euler-Darwin equation as in \eqref{Euler-Darwin-Lorentz fields}-\eqref{Euler-Darwin-Lorentz fields data}.
%\begin{align}
 % \label{Euler-Poisson alternative form}
  %  \partial_t \rho + \nabla \cdot (\rho(u-A)) &= 0,\\
   % \partial_t u + (u-A) \cdot \nabla u + u \cdot \nabla  A + \nabla (\frac{|A|^2}{2} + V) &= 0,\\
   %   -\Delta V &= \rho,  \label{V Euler Poisswell limit alternative}\\
 % -\Delta A &=  \rho (u - A),   \label{A Euler Poisswell limit alternative} \\
%    (\rho ,u)(x,0) &= (\rho^0,u^0) (x),
%\end{align}
%since $\Re i \langle{a},(\sigma \cdot B)a\rangle = 0$. 

\subsection{Main results}
\label{main results}

Define the normed space $X^s$ by 
\begin{equation}
    X^s := H^{s-1}(\mathbb{R}^3,\mathbb{C}^2) \times H^s (\mathbb{R}^3, \mathbb{R}^3), \qquad \| (a, u) \|_{X^s} :=   
  \|a\|_{H^{s - 1}} + \|u\|_{H^s} .
\label{def Xs}
\end{equation}

In the following we will always take $s>7/2$. Theorem \ref{thm:main1} is about the wellposedness and blow up of the Euler-Poisswell equation \eqref{Euler-Poisson}-\eqref{Euler Poisswell limit data} and
the Pauli-Poisswell-WKB equation \eqref{Madelung-Poisson}-\eqref{wkb data} as well as the Euler-Darwin equation \eqref{Euler-Darwin}-\eqref{Euler Poisswell limit data} and the Pauli-Darwin-WKB equation \eqref{Madelung-Poisson Darwin}-\eqref{wkb data darwin}. To our knowledge, these are the first wellposedness results on the Euler-Poisswell and Euler-Darwin equations.  The proof is contained in Proposition \ref{existence} (for Theorem \ref{thm:main result wellposedness}) and Proposition \ref{prop: blow up} (for Theorem \ref{thm:main result blow up}).

\begin{subtheorem}{theorem}
\label{thm:main1}
 \begin{theorem}\textbf{\emph{(Local wellposedness)}}\\
 \label{thm:main result wellposedness}
        Let $\psi^{\varepsilon,0}$ be of
  the form
  \begin{equation}
    \psi^{\varepsilon,0}_j (x) = a^{\varepsilon,0}_j (x) e^{\frac{i}{\varepsilon} S^{\varepsilon,0} (x)}, \quad j = 1,
    2.
  \end{equation}
  Let $s>7/2$ and let 
    \begin{equation}
    \|u^{\varepsilon,0} \|_{H^s} + \|a^{\varepsilon,0}  \|_{H^{s}} + \|u^{0} \|_{H^s} + \|a^{0}  \|_{H^{s}} \leq Q,
  \end{equation}
  for some $Q>0$ independent of $\varepsilon$. Assume that $(a^{\varepsilon,0}, u^{\varepsilon,0})$ converges to $(a^{0}, u^{0})$ in $X^{s-2}$. Then:
    %Let $s>7/2$ and let $(a^{\varepsilon,0},u^{\varepsilon,0}),(a^0,u^0) \in H^s$ such that  $(a^{\varepsilon,0}, u^{\varepsilon,0}) \to (a^{0}, u^{0})$ in $X^{s-2}$ as $\varepsilon\rightarrow 0$ where $X^{s}$ is defined by \eqref{def Xs}.
  %Then the following holds:
  
  \begin{enumerate}[label=(\roman*)] \item There exist respective unique local solutions 
  \begin{equation}(a, u)\in L^{\infty} ([0, T^{0}) , H^{s-1}) \times (L^{\infty} ([0, T^{0}) , H^{s}) \cap W^{1, \infty} ([0, T^{0}) , H^{s-1})),
  \end{equation} 
  to the Euler-Poisswell
  equation \eqref{Euler-Poisson}-\eqref{Euler Poisswell limit data} and the Euler-Darwin equation \eqref{Euler-Darwin}-\eqref{Euler Darwin limit data} with initial data $(a^0, u^0) \in H^{s} \times H^{s}$. Here $T^{0} > 0$ is the maximal (possibly infinite)  time of existence.
  
  \item There exist respective unique local solutions \begin{equation}
      (a^{\varepsilon}, u^{\varepsilon}) \in L^{\infty} ([0, T^{\varepsilon}) , H^{s-1}) \times (L^{\infty} ([0, T^{\varepsilon}) , H^{s})\cap W^{1, \infty} ([0, T^{\varepsilon}) , H^{s-1})),
  \end{equation}
  to the Pauli-Poisswell-WKB equation \eqref{Madelung-Poisson}-\eqref{wkb data} and the Pauli-Darwin-WKB equation \eqref{Madelung-Poisson Darwin}-\eqref{wkb data darwin} with initial
  data $(a^{\varepsilon, 0}, u^{\varepsilon, 0}) \in H^{s} \times H^{s}$. Here $T^{\varepsilon} > 0$ is the maximal (possibly infinite)  time of existence.
  
  \item There exists $\varepsilon_0 > 0$ such that $T^{\varepsilon}$ has a
  uniform positive lower bound $T$ for $0 < \varepsilon <
  \varepsilon_0$. In other words, there exist $\varepsilon_0 > 0$ and $T > 0$,
  such that
  \begin{equation}
    \inf_{0 < \varepsilon < \varepsilon_0} T^{\varepsilon} > T. \label{Teps lower bound}
  \end{equation}
  \end{enumerate}
\end{theorem}

\begin{theorem}\textbf{\emph{(Blow up alternative)}}\\
\label{thm:main result blow up}
  Under the assumptions of Theorem \ref{thm:main result wellposedness}, either the time of existence is global, or finite time blow up occurs. In the latter case it holds that:
  \begin{enumerate}[label=(\roman*)]
  \item If $T^0$ is finite, then the solution $(a, u)$ of \eqref{Euler-Poisson}-\eqref{Euler Poisswell limit data} blows up at $T^0$ such that
  \begin{equation} \lim_{t \rightarrow T^0 -} (\| a (t) \|_{L^{\infty}} + \| u (t) \|_{W^{1,
     \infty}}) =  \infty . \end{equation}
  \item If $T^{\varepsilon}$ is finite, then the solution $(a^{\varepsilon},
  u^{\varepsilon})$ of \eqref{Madelung-Poisson}-\eqref{wkb data} blows up at $T^{\varepsilon}$ such that
  \begin{equation} \lim_{t \rightarrow T^{\varepsilon} -} (\|a^{\varepsilon} (t)\|_{H^1} +\|a^{\varepsilon} (t)\|_{W^{1,
     \infty}} %+\|a^{\varepsilon} (t)\|_{W^{2, 3}} 
     + \| u^{\varepsilon} (t) \|_{W^{1, \infty}}) =  \infty
     . \end{equation}
  Furthermore, %although $\|u^{\varepsilon} (t) \|_{W^{1, \infty}} + \|a^{\varepsilon} (t)\|_{L^{\infty}}$ may not blow up at $T^{\varepsilon}$ like in the Euler-Poisswell equation, we can still show the following nearly singular behavior: 
  for $\varepsilon$ small enough, there exists a $K = K
  (\varepsilon)$, such that
  \begin{equation} \limsup_{t \rightarrow T^{\varepsilon} -} \|u^{\varepsilon} (t) \|_{W^{1, \infty}} + \|a^{\varepsilon} (t)\|_{L^{\infty}} \geq K. \end{equation}
  Here, $K$ goes to infinity as $\varepsilon$ goes to zero.
  \end{enumerate}
\end{theorem}
\end{subtheorem}

%\begin{remark}
%    Note that we don't explicitly assume any regularity for $(a^0,u^0)$. But due to the convergence in $X^{s-2}$ and the fact that there exists a weakly convergent subsequence of $(a^{\varepsilon,0}, u^{\varepsilon,0})$ as $\varepsilon \to 0$, it follows that $(a^0,u^0)$ has the same regularity as $(a^{\varepsilon,0}, u^{\varepsilon,0})$.
%\end{remark}

Theorem \ref{thm:main result semiclassical limit} is about the semiclassical limit and will be proved in Section \ref{section semiclassical limit}.

\begin{theorem}\textbf{\emph{Semiclassical limit}}\\
\label{thm:main result semiclassical limit}
  Let $T^0$ and $T^{\varepsilon}$ be the maximal times of existence of the solution $(a,u)$ of the Euler-Poisswell equation \eqref{Euler-Poisson}-\eqref{Euler Poisswell limit data}, respectively of the Euler-Darwin equation \eqref{Euler-Darwin}-\eqref{Euler Darwin limit data}, and the solution $(a^{\varepsilon},u^{\varepsilon})$ of the Pauli-Poisswell-WKB equation \eqref{Madelung-Poisson}-\eqref{wkb data}, respectively of the Pauli-Darwin-WKB equation \eqref{Madelung-Poisson Darwin}-\eqref{wkb data darwin}. Under the assumptions of Theorem \ref{thm:main result wellposedness} we have:
  \begin{enumerate}[label=(\roman*)]
  \item Suppose that $T< T^0$ is a
  lower bound of $T^{\varepsilon}$ satisfying \eqref{Teps lower bound}. We
  have the following semiclassical limit
  \begin{equation}
    \| (a^{\varepsilon} - a, u^{\varepsilon} - u) \|_{L^{\infty} ([0, T], X^{s
    - 2})} \underset{\varepsilon \rightarrow 0}{\longrightarrow} 0.
    \label{semiclassical limit eq}
  \end{equation}
  \item The Wigner transform $f^{\varepsilon}$ (cf. Appendix \ref{sec:Wigner}) converges towards the monokinetic Wigner measure
  \begin{equation} f^{\varepsilon}  (x, \xi, t) \underset{\varepsilon \rightarrow
     0}{\rightharpoonup} f (x, p, t) = \rho (x,t) \delta (p - u (x,t))  \text{ in }
     \mathcal{A}', \quad t \in [0,T) \end{equation}
  where $p=\xi-A(x,t)$ and $\mathcal{A}'$ is the dual of $\mathcal{A}$ defined by \eqref{algebra A}. In particular, f solves the Vlasov-Poisswell equation \eqref{eq:vlasov_limit}-\eqref{eq:data_limit}, respectively the Vlasov-Darwin equation \eqref{eq:vlasov_darwin_limit}-\eqref{eq:data_limit Darwin} in $\mathcal{D}'$ with initial
  data
  \begin{equation}
    f (x, p,0) = \rho^0 (x,t) \delta (p - u^0 (x,t)) .
    \label{eq:vlasov_initial_data}
  \end{equation}
  \item For the macroscopic densities $\rho^{\varepsilon}$ and $J^{\varepsilon}$ it holds that
  \begin{align*}
    \rho^{\varepsilon}   = |a^{\varepsilon}  |^2 &
    \underset{\varepsilon \rightarrow 0}{\longrightarrow} \rho  = |a  |^2 & \text{in } L^{\infty} ([0, T], H^{s - 3} (\mathbb{R}^3)),\\
    J^{\varepsilon}  = \varepsilon w^{\varepsilon} + \varepsilon
    v^{\varepsilon} + \rho^{\varepsilon}  (u^{\varepsilon} - A^{\varepsilon})
    & \underset{\varepsilon \rightarrow 0}{\longrightarrow} J =  \rho (u  - A) & \text{in } L^{\infty} ([0, T], H^{s - 3} (\mathbb{R}^3)),
  \end{align*}
  
  %where $A$ is given by $- \Delta A = \rho (u  - A)$.
  
  \item If $s \geq 6$,
  then the asymptotic behavior of the maximal time of existence $T^{\varepsilon}$
  satisfies
  \begin{equation} \label{eq:Teps T0}\liminf_{t \rightarrow T^{\varepsilon} -} T^{\varepsilon} \geq T^0 .
  \end{equation}
  \end{enumerate}
\end{theorem}
\begin{remark}
  In this paper we consider the three dimensional case. The results can be generalized to higher dimensions by replacing the assumption $s>7/2$ with $s>d/2 + 2$. Note that the Poisson equation for $V^{\varepsilon}$ is no longer equivalent to the Hartree term in dimensions other than three.
\end{remark}

\begin{remark}
  The WKB ansatz is local in time with a singular, monokinetic Wigner measure. Global results can be obtained if the Wigner measure is regular enough: For the linear case a general survey can be found in \cite{GMMP}. For the Schrödinger-Poisson equation see \cite{LionsPaul, MarkowichMauser} and for the Pauli-Poisson equation \cite{moller2023poisson} (the global-in-time semiclassical limit of the Pauli-Poisswell and Pauli-Darwin equations is ongoing work \cite{MaMo23}). The downside is that the convergence is usually only weak and the limit equation only holds in the sense of distributions.
\end{remark}

\begin{remark}For the mathematical techniques we were inspired by three different papers in which the semiclassical limit of different nonlinear Schrödinger equations was discussed. Our energy estimates are a generalization to the spin-magnetic case of the work by E. Grenier in \cite{10.2307/118716}, where the semiclassical limit for the scalar nonlinear Schrödinger equation with nonlinearity $f(|\psi|^2)\psi$ and $A^{\varepsilon},V^{\varepsilon} \equiv 0$ was treated. We also base our considerations on the papers by P. Zhang \cite{2003Wigner} for the semiclassical limit of the Wigner-Poisson equation and by T. Alazard and R. Carles  \cite{carles2007} for the semiclassical limit of the Schrödinger-Poisson equation with doping profile in $d\geq 3$ dimensions.
\end{remark}

\subsection{Asymptotic hierarchy of relativistic Euler equations}
\label{physical motivation}
The \textbf{Euler-Maxwell equation} \cite{germain2014, guo2016} is given by, 
\begin{align}
  \label{Euler-Maxwell}
    &\partial_t \rho +  \nabla \cdot (\rho u)
     = 0\\
    &\partial_t (\rho u) + \nabla\cdot(\rho u \otimes u )  = \rho(E + u\times B),\\
    &\partial_tE -\nabla \times B = -\rho u, \quad \partial_t B + \nabla \times E = 0, \\
    &\nabla \cdot E = \rho, \quad \nabla\cdot B = 0. \\
    &(\rho,u,E,B) (x,0) = (\rho^0 , u^0 , E^0, B^0)(x) \label{eq:eulermaxwell data}
\end{align}
The \textbf{relativistic Euler-Maxwell equation} \cite{guo2014global}, where the relativistic Euler equation is coupled to Maxwell's equations, is given by
\begin{align}
  \label{rel unipolar-Euler-Maxwell}
    &\partial_t (\gamma(u){\rho}) +  \nabla \cdot (\rho \gamma(u) u)
     = 0\\
    &\partial_t(\rho\gamma(u)u) + u\cdot \nabla (\rho\gamma(u) u) = \rho(E +  \gamma(u)u\times B), \label{rel unipolar-Euler-Maxwell u}\\
    &\partial_t E -\nabla \times B = - \rho \gamma(u)u, \quad \partial_t B + \nabla \times E = 0, \\
    &\nabla \cdot E = \rho, \quad \nabla\cdot B = 0. \label{rel unipolar EM divergence E,B}\\
    &(\rho,u,E,B) (x,0) = (\rho^0 , u^0 , E^0, B^0)(x)\label{rel unipolar EM data}
\end{align}
where the gamma factor of special relativity is defined as
\begin{equation}
    \label{eq:gamma}
    \gamma(u)= \frac{1}{\sqrt{1-{\frac{1}{c^2}}|u|^2 }}.
\end{equation}
These models describe the fully relativistic self-consistent interaction with the electromagnetic field. We have the following diagram representing the asympotics between self-consistent models in relativistic quantum physics and der classical limits. The models considered in this paper are written in bold:\\

{\begin{tabular}{ccc}
    {\fbox{{ {Dirac}-Maxwell}}}
    & {$\stackrel{\hbar \rightarrow 0}{\longrightarrow}$}
    & {\fbox{\begin{varwidth}{\textwidth}\centering rel. Vlasov w/ Lorentz force - Maxwell \\ Euler w/ Lorentz force - Maxwell \end{varwidth}}} \\[3mm]
  {$\downarrow$}
    &
    & {$\downarrow$} \\[3mm]

  {\fbox{ {\bf Pauli-Darwin $O(1/c)$}}}
    & {$\stackrel{\hbar \rightarrow 0}{\longrightarrow}$}
    & {\fbox{\begin{varwidth}{\textwidth}\centering  \textbf{Vlasov w/ Lorentz force - Darwin} \\ {\textbf{Euler w/ Lorentz force - Darwin}} \\{$O(1/c)$}\end{varwidth}}} \\[3mm] \\

   {\fbox{\bf{ {Pauli}-Poisswell $O(1/c)$}}}
    & {$\stackrel{\hbar \rightarrow 0}{\longrightarrow}$}
    & {\fbox{\begin{varwidth}{\textwidth}\centering  \textbf{Vlasov w/ Lorentz force - Poisswell} \\ {\textbf{Euler w/ Lorentz force - Poisswell}} \\{$O(1/c)$}\end{varwidth}}} \\[6mm]

 {$\downarrow$}
    &
    & {$\downarrow$} \\[3mm]

      {\fbox{ Pauli-Poisson}}
    & {$\stackrel{\hbar \rightarrow 0}{\longrightarrow}$}
    & {\fbox{\begin{varwidth}{\textwidth}\centering  Vlasov w/ Lorentz force - Poisson \end{varwidth}}} \\[3mm]
{\fbox{{magn. Schrödinger-Maxwell}}}
    & {$\stackrel{\hbar \rightarrow 0}{\longrightarrow}$}
    & {\fbox{\begin{varwidth}{\textwidth}\centering Vlasov w/ Lorentz force
    - Maxwell \end{varwidth}}} \\[3mm]
    {\fbox{{magn. Schrödinger-Poisson}}}
    & {$\stackrel{\hbar \rightarrow 0}{\longrightarrow}$}
    & {\fbox{\begin{varwidth}{\textwidth}\centering Vlasov w/ Lorentz force - Poisson \end{varwidth}}} \\[3mm]
  {$\downarrow$}
    &
    & {$\downarrow$} \\[3mm]
  {$\stackrel{c \rightarrow \infty}{\longrightarrow}$}
    &
    & {$\stackrel{c \rightarrow \infty}{\longrightarrow}$} \\[3mm]
  $\downarrow$
    &
    & $\downarrow$ \\[3mm]
  {\fbox{  {Schr\"odinger}-Poisson}}
    & {$\stackrel{\hbar \rightarrow 0}{\longrightarrow} $} 
    & {\fbox{\begin{varwidth}{\textwidth}\centering Vlasov-Poisson \\ Euler-Poisson \end{varwidth}}}
    \\
\end{tabular}}\\
\bigskip

Since the Pauli-Poisswell equation is the semi-relativistic $O(1/c)$ approximation of the Dirac-Maxwell equation \cite{masmoudi2001selfconsistent} and the Euler-Maxwell equation is the semiclassical limit of the Dirac-Maxwell equation, it is natural that the Euler-Poisswell equation \eqref{Euler-Poisswell-Lorentz fields}-\eqref{Euler-Poisswell-Lorentz fields data} arises as the $O(1/c)$ semi-relativistic approximation of the Euler-Maxwell equation and as the semiclassical limit of the Pauli-Poisswell equation. The Pauli-Darwin and the Euler-Darwin equation are also $O(1/c)$ since the  $O(1/c^2)$ Darwin approximation of Maxwell's equations is coupled to the $O(1/c)$ Pauli equation.

Therefore they are correct semi-relativistic self-consistent $O(1/c)$ models for a plasma of charged particles with self-interaction with the electromagnetic field, cf. \cite{moller2023euler} where the full hierarchy from the (relativistic) Euler-Maxwell equation, the semi-relativistic Euler-Darwin and Euler-Poisswell equations to the non-relativistic Euler-Poisson equation was presented.

Note that \eqref{Euler-Poisswell-Lorentz fields}-\eqref{Euler-Poisswell-Lorentz fields data} includes the $O(1/c^2)$ term $u\times B$, i.e.  the magnetic part of the Lorentz force. This is due to the $O(1/c^2)$ term $|A|^2$ in the magnetic Laplacian in \eqref{eq:PPW_Pauli_scaled}. In a formal $O(1/c)$ asymptotic expansion of the Euler-Maxwell equation \eqref{Euler-Maxwell}-\eqref{eq:eulermaxwell data}, only the electric part of the Lorentz force remains and the magnetic field is decoupled from the fluid dynamics \cite{moller2023euler}.

\begin{remark}{Caveat on "Darwin":} Equations \eqref{eq:PD_V}-\eqref{eq:coulomb gauge} are referred to as \emph{Darwin's equations} and are the $O(1/c^2)$ approximation of Maxwell's equations. The \emph{Poisswell equations} \eqref{eq:PPW_V}-\eqref{eq:lorenz gauge} are the $O(1/c)$ approximation of Maxwell's equations  \cite{besse2007, krause2007unified, moller2023euler, pallard2006, seehafer2009, shiroto2023improved}. 

On the other hand, the \emph{Darwin term} in relativistic quantum mechanics refers to a second order term in the semi-relativistic approximation of the Dirac equation and is related to the \emph{Zitterbewegung} \cite{itzykson2012quantum, mauser1999rigorous, mauser2000semi}. The Pauli equation as the the first order semi-relativistic approximation of the Dirac equation does \emph{not} include the Darwin term. The complete self-consistent second order (in $1/c$) approximation of the Dirac-Maxwell equation, including second order corrections of the Pauli equation such as the Darwin term and spin-orbit-coupling is presented in \cite{GeMM23} and \cite{itzykson2012quantum, moller2023thesis}. 
\end{remark}

In the non-relativistic limit $c\rightarrow \infty$  the Euler-Maxwell equation \eqref{Euler-Maxwell}-\eqref{eq:eulermaxwell data} reduces to the \textbf{Euler-Poisson equation} \cite{peng2007, yang2009}, given by
\begin{align}
  \label{Euler-Poisson no A with rho u}
    \partial_t \rho +  \nabla \cdot (\rho u)
     &= 0\\
    \partial_t u +  u \cdot \nabla u   &= - \nabla
     V,\label{eq:eulerpoisson u}\\
      - \Delta V &=  \rho  \label{V Euler Poisson}, \\
    (\rho,u) (x,0) &= (\rho^0,u^0) (x).
\end{align}
The Euler-Poisson equation arises as the semiclassical limit of the Schrödinger-Poisson equation with WKB initial data \cite{carles2007,Carl07,golse2022,liu2002, masaki2011, 2003Wigner, Zhan08}.% In \cite{golse2022} it was shown that the bosonic $N$-particle Schrödinger dynamics with Coulomb interaction converge to the Euler-Poisson equation in the combined semiclassical and mean field limit $\hbar + 1/N \rightarrow 0$. %, if the first marginal of the initial data has a monokinetic Wigner measure (for the definition cf. Section \ref{sec:Wigner}) which is equivalent to the WKB formulation.

Depending on the sign in front of the term $\nabla V$ in \eqref{eq:eulerpoisson u} the Euler-Poisson equation is \emph{repulsive} (with a negative sign) or \emph{attractive} (with a positive sign). The repulsive Euler-Poisson equation is a non-relativistic self-consistent model for a plasma of charged particles, self-interacting with the electric field via the Poisson equation \eqref{V Euler Poisson}. The Euler-Poisson equation is a self-consistent model for non-relativistic gravitational interaction, e.g. gaseous stars. The repulsive case is less delicate than the attractive case due to dispersive effects of the electric field \cite{guo1998, guopausader, hadzic2019}. %\cite{guo1998, guopausader} for the global existence of irrotational flows in $\mathbb{R}^3$, \cite{hadzic2019} for global solutions for the repulsive and attractive case.  However blow up \cite{yuen2011} and non-existence   \cite{perthame1990} are also possible.

%%%%%%%%%%%%%%%%%%%%%%%%%%%%%%%%%%%%%%%%%%%%%%%%%%%%%%%%%%%%%%%%%%%%%%%%%%%
\subsection{Organization of the article} 

In section \ref{section elliptic} we prove elliptic estimates for the Poisson equations. In section \ref{4} prove an a priori estimate (Proposition \ref{a priori opt}). In section \ref{5}, we prove existence and uniqueness using the a priori estimate and a fixed point argument. In section \ref{section blow up} we  study the behaviour at blow up time. In section \ref{section semiclassical limit} we prove the semiclassical limit of the Pauli-Poisswell-WKB equation \eqref{Madelung-Poisson}-\eqref{wkb data} towards the Euler-Poisswell equation. 

\section{Notations and useful identities}
\label{sec:notations}

The quantities $a^{\varepsilon} = (a_1^{\varepsilon},a_2^{\varepsilon})^T$ and $a = (a_1,a_2)^T$  always denote 2-spinors. On the other hand, $u^{\varepsilon}$, $u$, $v^{\varepsilon}$, $v$, $w^{\varepsilon}$ and $w$ denote vectors in $\mathbb{R}^3$.

For $\mu, \mu_1, \mu_2 >0$ and
\begin{equation*}
a \in L^{\infty} ([0, T) , H^s ), \quad u \in W^{1,
\infty} ([0, T) , H^s)
\end{equation*}
we define the following
functionals. 
\begin{align}
%  E_s  (t) &:= \| (a, u) (t) \|_{X^s},  \label{Estilde def} \\
  E_s^{\mu}(t) &:= \|(a,u)\|_{X^s} + \mu \varepsilon \|a (t) \|_{H^s},   %\label{Es def}
  \\
  E_s^{\mu_1, \mu_2}  (t) &:= E_s^{\mu_1} (t) + \mu_2 \|
  \partial_t u (t) \|_{H^{s - 1}},  \label{E2 def}\\
  M (t) &:= 1 + \|u (t) \|_{W^{1, \infty}} + \|a (t) \|_{L^{\infty}} + \varepsilon \|a (t)
  \|_{H^1} + \varepsilon\|a (t) \|_{W^{1, \infty}} %+ \varepsilon\|a (t) \|_{W^{2,3}},  
  \\
  N  (t) &:= \sup_{0 \leq \tau \leq t} M (\tau) . 
  \label{N def}
\end{align}
where
\begin{equation*}
    \| (a, u) \|_{X^s} :=   
  \|a\|_{H^{s - 1}} + \|u\|_{H^s} .
\end{equation*} 
\subsection{Inequalities}
For two quantities $A,B$, we write $A \lesssim B$ if there exists a constant $C$ such that $A \leq CB$. 
We write $A \sim B$ if $A \lesssim B$ and $B \lesssim A$.  
    If $C({M_1, \ldots, M_n})$ depends on
    constants $M_1, \ldots M_n$ we write $A \lesssim_{M_1, \ldots, M_n} B$
    for $A \lesssim B$.  We omit the dependence on the regularity parameter $s$ and use $\lesssim$ instead of $\lesssim_s$. 
    % for simplicity, where $s$ is the required regularity in $H^s$.
    %For a $m\times n$ matrix $A$ and a $n$-dimensional vector $x$ we use summation convention, i.e.
    %\begin{equation*}
    %     A_i^j x_j := \sum_{j=1}^n A_{i j} x_j.
    %\end{equation*}
    \subsection{Fourier multipliers and Sobolev spaces}

For $s$ a real number, the operator $D^s$ is the Fourier multiplier with symbol $|\xi|^s$.

The homogeneous Sobolev space $\dot{W}^{s,p}$  of order $s$ is defined by the norm $\|D^s f\|_{L^p}$ and  the inhomogeneous Sobolev space $W^{s,p}$  of order $s$ defined by the norm $ \|\langle D \rangle^s f \|_{L^p}$ where $\langle \xi \rangle = (1+|\xi|^2)^{1/2}$. If $p=2$, we denote $\dot H^s$ and $H^s$ instead of $\dot W^{s,p}$ and $W^{s,p}$.
    %For two vector-valued functions $f,g:\mathbb{R}^d\to \mathbb{R}^d$ we use $g\cdot \nabla f$ to denote the vector $\sum_j(\partial_i f_j g_j)$, $i=1,...,d$. 
   % For $f \inH^s$ and $\alpha = (\alpha_1, \ldots, \alpha_s)$, we denote
    %\begin{equation*} \partial^{\alpha} f (x) := \partial^{\alpha_1} \cdots  \partial^{\alpha_l} f (x) . \end{equation*}
       %Let $V,W$ be vector-valued functions with components $V_j,W_j$. By $V \otimes W$ we denote the matrix
     %  \begin{equation*}
      %     (V \otimes W)_{ij} = V_i W_j.
      % \end{equation*}
   % For a matrix $X \in \mathbb{R}^{m \times n}$, we will define the magnitude, the $L^2$ norm, the $H^s$ norm and the $W^{s, \infty}$ norm as
    %\begin{align*}
      %| X |^2 &=  \sum_{1 \leq i \leq m, 1 \leq j \leq n} | X_{i j} |^2,& \| X \|_{L^2} &= \| (| X |) \|_{L^2}, \\
      %\| X \|_{H^s} &=  \sum_{|\alpha| < n} \| \partial^{\alpha} X \|_{L^2}, & \| X \|_{W^{s, \infty}} &=  \sup_{\substack{|\alpha| < s,\\ 1 \leq i \leqm,\\ 1 \leq j \leq n}} \| \partial^{\alpha} X_{i j} \|_{L^{\infty}} .
   % \end{align*}
   % For vectors $v \in \mathbb{R}^d$ we can define the respective norms regarding $v$ as an $1 \times d$ matrix.
    
   % We will often abbreviate $\|\cdot \|_{L^p}=\|\cdot \|_{p}$. \\

\subsection{Linear and multilinear operations}
If $v,w \in \mathbb{C}^2$, their inner product is defined as
\begin{equation*}
    \langle v , w\rangle_{\mathbb{C}^2} = \langle v , w\rangle = \overline{v_1} w_1 + \overline{v_2} w_2 .
\end{equation*}
By contrast, if $a,b \in \mathbb{C}^3$, we denote $a \cdot b$ for%\jm{I think $a,b$ should be in $\mathbb{R}^3$} \pg{right, but if you think of $a \dot \sigma$?}
$$
a \cdot b = \sum_{i=1}^3 a_i b_i.
$$
The $L^2$ inner product of $v$ and $w$ valued in $\mathbb{C}^2$ is denoted by
$$
(v,w)_{L^2} = (v,w) = \int \langle v,w \rangle \dd x.
$$
We will also denote
\begin{equation}
a \cdot \nabla = \sum_{i=1}^3 a_i \partial_i, \qquad a \cdot \sigma = \sum_{i=1}^3 a_i \sigma_i. 
\end{equation}
Here, the Pauli matrices are given by
\begin{align}
    \sigma_1 = \begin{pmatrix}
    0 & 1 \\ 1 & 0
    \end{pmatrix}, && 
     \sigma_2 = \begin{pmatrix}
    0 & -i \\ i & 0
    \end{pmatrix}, &&
     \sigma_3 = \begin{pmatrix}
    1 & 0 \\ 0 & -1
    \end{pmatrix}.\label{eq:pauli_matrices}
\end{align}
They obey the following product laws
\begin{equation}
\label{productlaws}
\begin{split}
& \sigma_i^2 = \operatorname{Id} \qquad \mbox{for $i=1,2,3$}\\
& \sigma_i \sigma_j =  i \sum_{k=1}^3 \varepsilon_{ijk} \sigma_k \qquad \mbox{for $i \neq j$},
\end{split}
\end{equation}
where $\varepsilon_{ijk}$ denotes the Levi-Civita symbol. 

\subsection{Useful identities}
\label{sec:useful_identities}

%We define the magnetic gradient \begin{equation}
%\nabla_A = \nabla-iA = (\partial_k -i A_k)_{k=1,2,3},
%\end{equation}
%the magnetic Laplacian
%\begin{equation}
%     \Delta_A = (\nabla_A)^2 = \sum_{k=1}^3 (\partial_k - iA_k)^2 = \Delta-2iA\cdot \nabla  - i(\text{div} A) -|A|^2
%\end{equation}
%and the spin-magnetic Laplacian
%\begin{equation}
% (\sigma \cdot \nabla_A)^2 = \left( \sum_{k=1}^3 \sigma_k (\partial_k - iA_k) \right)^2. 
%\end{equation}
%It can be written as
%$$
% (\sigma \cdot \nabla_A)^2 = \Delta_A + \sigma \cdot B.
%$$
The Pauli Hamiltonian can be rewritten as
\begin{equation}
    -\frac{1}{2}(\varepsilon \nabla-{i}A^{\varepsilon})^2 -\frac{\varepsilon}{2} (\sigma \cdot B^{\varepsilon})  = -\frac{1}{2}(\sigma \cdot (\varepsilon \nabla-{i}A^{\varepsilon}))^2.
\end{equation}
To prove this identity, we proceed as follows:
\begin{align*}
(\sigma(\varepsilon \nabla-{i}A^{\varepsilon}))^2 & = \sum_{k,\ell} \sigma_k \sigma_\ell (\varepsilon\partial_k - iA^{\varepsilon}_k)(\varepsilon\partial_\ell - iA^{\varepsilon}_\ell) \\
& = \sum_k (\varepsilon\partial_k - iA^{\varepsilon}_k)^2 + i \sum_{k,\ell,m} \epsilon_{k\ell m} \sigma_m (\varepsilon\partial_k - iA^{\varepsilon}_k)(\varepsilon\partial_\ell - iA^{\varepsilon}_\ell) \\
& = (\varepsilon \nabla-{i}A^{\varepsilon})^2 + \varepsilon\sum_{k,\ell,m} \epsilon_{k\ell m} \sigma_m \partial_k A^{\varepsilon}_\ell = (\varepsilon \nabla-{i}A^{\varepsilon})^2 + \varepsilon\sigma \cdot B^{\varepsilon}.
\end{align*}
Here, we used \eqref{productlaws} in the second equality and the cancellations following from the antisymmetry of the Levi-Civita symbol in the third equality.

Note that the spin-magnetic Laplacian (just like the magnetic Laplacian) is symmetric with respect to the $L^2$ scalar product:
\begin{equation}
    \int_{\mathbb{R}^3} \langle f,(\sigma \cdot (\varepsilon \nabla-{i}A^{\varepsilon})^2 g\rangle \dd x = \int_{\mathbb{R}^3} \langle(\sigma \cdot (\varepsilon \nabla-{i}A^{\varepsilon})^2 f, g\rangle \dd x.
\end{equation}

%The Pauli Hamiltonian
%\begin{equation}
%    H = -\frac{1}{2} (\mathbf{\sigma} \cdot \nabla_A)^2 + V = -\frac{1}{2} \Delta_A + V  -\frac{1}{2} (\mathbf{\sigma} \cdot B) 
%\end{equation}
%can be rewritten as a perturbation $K$ of the Laplacian $\Delta$, i.e.
%\begin{equation}
%    Hu = -\frac{1}{2}\Delta u + iK u,
%\end{equation}
%where
%\begin{equation}\label{eq:def K}
%K u = A\cdot \nabla u +  \frac{1}{2}(\text{div}A)u -i\frac{1}{2}|A|^2 u -i V u +\frac{i}{2}(\sigma \cdot B) u\end{equation}
%(this will be needed for the local wellposedness). 

Finally, the current density can be written as
\begin{align}
    J^{\varepsilon} = \Im \langle \sigma \psi, (\sigma \cdot (\varepsilon \nabla-{i}A^{\varepsilon}) \psi \rangle = \Im \langle \sigma_j \psi, \sum_{k=1}^3 \sigma_k(\varepsilon\partial_k -iA^{\varepsilon}_k) \psi \rangle. \label{eq:J_k}
\end{align}
Indeed, using once again the properties of Pauli matrices,
\begin{align*}
\operatorname{Im} \langle \sigma_j \psi, \sum_k \sigma_k (\partial_k - i A^{\varepsilon}_k) \psi \rangle 
& =  \operatorname{Im} \langle \psi , \sum_k \sigma_j \sigma_k (\partial_k - i A^{\varepsilon}_k) \psi \rangle \\
& =  \operatorname{Im} \langle \psi ,  (\partial_j - i A^{\varepsilon}_j) \psi \rangle + \sum_{k,l}  \operatorname{Im} \langle \psi , \epsilon_{jkl} i \sigma_l (\partial_k - i A^{\varepsilon}_k) \psi \rangle.
\end{align*}
Furthermore,
\begin{align*}
\sum_{k,l}  \operatorname{Im} \langle \psi , \epsilon_{jkl} i \sigma_l (\partial_k - i A^{\varepsilon}_k) \psi \rangle & = \sum_{k,l}  \operatorname{Re} \langle \psi , \epsilon_{jkl}  \sigma_l (\partial_k - i A^{\varepsilon}_k) \psi \rangle = \sum_{k,l} \epsilon_{jkl}  \operatorname{Re} \langle \psi , \sigma_l \partial_k \psi \rangle \\
& = \frac{1}{2} \sum \epsilon_{jkl} \partial_k  \langle \psi , \sigma_l \psi \rangle = \frac{1}{2} \left[ \nabla \times \langle \psi , \sigma \psi \rangle \right]_j.
\end{align*}

\section{Elliptic estimates}\label{section elliptic}
From now on we mostly omit the superscript $\varepsilon$ for readability.
\begin{lemma}
  \label{thm:estimates_poisson}Let $s >7/2 $ and $(a,u) \in X^s$. Then the boundary value problems on $\mathbb{R}^3$
\begin{align}
    -\Delta A = J,  && |A(x)| \rightarrow \infty \text{ as } |x|\rightarrow 0
\end{align}
\begin{align}
    -\Delta A = \mathbb{P}J,  && |A(x)| \rightarrow \infty \text{ as } |x|\rightarrow 0
\end{align}
with \begin{equation}
  J = \varepsilon w + \varepsilon v + \rho (u - A),
\end{equation}
where
\begin{align}
  v := \frac{1}{2} \nabla \times \langle {a} ,\sigma  a\rangle  
  &&
  w :=  \Im\langle{a} ,\nabla a \rangle, 
\end{align}
each have a unique solution $A \in \dot{H}^1$. Moreover, the solutions $V$ and $A$ satisfy
  \begin{align}\label{nabla V Hs bound}
    \| \nabla V\|_{H^s} &\lesssim  M (t) \|a\|_{H^{s - 1}},
    \\ \| \nabla A\|_{H^s} &\lesssim  M (t)^{2 s} E_s^1 (t),  \label{VAHs}
    \\ \|  A\|_{W^{1, \infty}} &\lesssim  M (t)^5 .  \label{VALinf}
  \end{align}
  where $E^1_s(t)$ is given by \eqref{Es def} with $\mu=1$ and $M(t)$ is given by \eqref{M def}.
\end{lemma}

Note that there is no estimate for $\|A\|_{L^2}$.

\begin{proof} Since $\mathbb{P}$ is bounded on $L^p$, $1<p<\infty$ we may restrict ourselves to $-\Delta A = J$. The problem  can be written as
\begin{align}
    LA(x) &= f(x), \quad x \in \mathbb{R}^3, \quad |A(x)| \rightarrow 0 \text{ as }|x|\rightarrow \infty \\
    f(x) &= -\varepsilon w(x)-\varepsilon v(x) -\rho(x)u(x)
\end{align}
where $L$ is defined as
\begin{equation}
    L := \Delta - \rho.
\end{equation}

Defining the bilinear form $B$ as
\begin{equation}
    B[A,A'] = \int \nabla A \cdot \nabla A' \dd x + \int A\cdot A'\rho \dd x,
\end{equation}
it is bounded and coercive on $\dot{H}^1(\mathbb{R}^3)$ if we can bound
$$
\int A \cdot A' \rho \dd x \lesssim \| \nabla A \|_{L^2} \| \nabla A' \|_{L^2}, 
$$
which is the case if $a \in L^3 \subset H^1$ by Sobolev embedding. As for the form $(f,\cdot)$, it is bounded on $\dot{H}^1(\mathbb{R}^3)$ if $f \in L^{6/5} \subset \dot{H}^{-1}$, which is also ensured by $(a,u) \in X^1$, as follows from H\"older's inequality and the Sobolev embedding theorem.

Therefore, the Lax-Milgram theorem (or the Riesz representation theorem) gives a unique solution $A \in \dot H^1$ if $(a,u) \in X^1$.\\

\emph{Estimate of $\| \nabla V\|_{H^s}$.}
  Using H{\"o}lder and Sobolev's inequality,
  \begin{equation}
      \| \nabla V\|_{L^2} \lesssim \| \nabla^2 V\|_{L^{6 / 5}}  \lesssim \||a|^2 \|_{L^{6 / 5}} \lesssim (\|a\|_{L^2}^{5 / 6} \|a\|_{L^{\infty}}^{1 / 6})^2 \lesssim M (t) \| a \|_{L^2} . \label{V1}
  \end{equation}
  Using the Kato-Ponce inequality yields
  \begin{equation}
      \| \nabla^2 V\|_{H^{s - 1}} \lesssim \||a|^2 \|_{H^{s - 1}} \lesssim \|a\|_{H^{s - 1}} \|a\|_{L^{\infty}}  \lesssim M (t) \|a\|_{H^{s - 1}} .\label{nabla2 V Hs bound}
  \end{equation}
  Combining  
  \eqref{V1} and \eqref{nabla2 V Hs bound} yields \eqref{nabla V Hs bound}. \\
  
  \emph{Estimate of $\| A\|_{W^{1,\infty}}$.}  By 
  the Gagliardo-Nirenberg inequality, 
  \begin{align*}
    \|A\|_{L^{\infty}} \lesssim \|A\|_{L^6}^{\alpha_1} \|A\|_{\dot{W}^{2,
    6}}^{1 - \alpha_1}, &&
    \| \nabla A\|_{L^{\infty}}  \lesssim  \|A\|_{L^6}^{\alpha_2}
    \|A\|_{\dot{W}^{2, 6}}^{1 - \alpha_2} .
  \end{align*}
  For some $\alpha_1, \alpha_2 \in (0, 1)$. Then
  \begin{equation*} \|A\|_{W^{1, \infty}} \lesssim \|A\|_{L^6} + \|A\|_{\dot{W}^{2, 6}} \end{equation*}
  From $-\Delta A = J$,
  \begin{align*}
    \|A\|_{\dot{W}^{2, 6}} %& \lesssim  \|J\|_{L^6}\\&
    &\lesssim  \varepsilon \|w \|_{L^6} + \varepsilon \| v\|_{L^6} + \|
    \rho u\|_{L^6} + \| \rho A\|_{L^6} \\ 
  %\end{align*}
  %From the definition of $v$ and $w$ \eqref{w def} and Hölder it follows that
  %\begin{align*}
    %\|A\|_{\dot{W}^{2, 6}}  %&\lesssim \varepsilon \|a\|_{L^6} \|a \|_{W^{1,\infty}} +\| \rho \|_{L^6} \| u\|_{L^{\infty}} + \| \rho \|_{L^{\infty}} \|A\|_{L^6}\\
    &\lesssim  \varepsilon \|a\|_{L^6} \|a \|_{W^{1, \infty}} +\|a\|_{L^{12}}^2 \|
    u\|_{L^{\infty}} + \|a\|_{L^{\infty}}^2 \|A\|_{L^6}\\
     & \lesssim  M (t)^2 (1 +\|\nabla A\|_{L^2}) .
  \end{align*}
  %In the last step, we use \eqref{alp ori} and the Sobolev inequality for $\|A\|_6$.
    %Combined with the Sobolev inequality, this yields
 % \begin{equation}
  %  \|A\|_{L^6} \lesssim \| \nabla A\|_{L^2}, \label{L6H1}
  %\end{equation}
  %which leads to
  %\begin{equation*}
   % \|A\|_{W^{1, \infty}} \lesssim M (t)^2 (1 +\| \nabla A\|_{L^2}). \label{VAJ est}
  %\end{equation*}
  Multiplying
  $-\Delta A = J$ by $A$ and integrating by parts yields
  \begin{align*}
    \| \nabla A\|_{L^2}^2 &=  \int (\varepsilon (v - w) + \rho (u -
    A)) A\dd x \leq  \int (\varepsilon (v - w) + \rho u) A\dd x\\
    &\lesssim  \| \varepsilon (v - w) + \rho u \|_{L^{6 / 5}} \|A\|_{L^6} .
  \end{align*} 
    Using Sobolev's and Hölder's inequalities we obtain
  \begin{align}
    \| \nabla A\|_{L^2} &\lesssim \| \varepsilon (v - w) + \rho u \|_{L^{6 / 5}} \nonumber \\ %&\lesssim  \| \rhou\|_{L^{6 / 5}} + \varepsilon \|v - w\|_{L^{6 / 5}} \nonumber \\
    & \lesssim \varepsilon
    \|a\|_{L^3}  \| \nabla a\|_{L^2}+ \||a|^2  \|_{L^{6 / 5}} \|u  \|_{L^{\infty}} \nonumber \\
    & \lesssim  M (t)^3\label{eq:nabla A L2}
  \end{align}
  %By Hölder's inequality,
  %\begin{align*}
   % \| \varepsilon (v - w) + \rho u\|_{L^{6 / 5}} & \leq  \| \rho u\|_{L^{6 / 5}} + \varepsilon \|v - w\|_{L^{6 / 5}}\\
  %  & \lesssim  \||a|^2  \|_{L^{6 / 5}} \|u  \|_{L^{\infty}} + \varepsilon  \|a\|_{L^3}  \| \nabla a\|_{L^2}\\
   % & \lesssim  M (t)^3 ,
 % \end{align*}
  Thus,
  \begin{align*} \|A\|_{W^{1, \infty}} &\lesssim M (t)^2 (1 + \| \varepsilon (v - w) + \rho
     u \|_{L^{6 / 5}})\\
     &\lesssim M (t)^5.\end{align*} 
  
  \emph{Estimate of $\|\nabla A\|_{H^s}$.} 
  We have
  \begin{align}
     \| \nabla A\|_{L^2} \lesssim \| \varepsilon (v - w) + \rho u\|_{L^{6 / 5}} & \leq  \| \rho
    u\|_{L^{6 / 5}} + \varepsilon \|v - w\|_{L^{6 / 5}}\nonumber \\
    & \lesssim  \||a|^2  \|_{L^3} \|u  \|_{L^2} + \varepsilon \|a\|_{L^3} 
    \| \nabla a\|_{L^2} \nonumber \\
    & \lesssim  M (t)^2 E_s^1 (t) \label{VAL2}.
  \end{align}
  %In the last step we used \eqref{alp ori}. 
  %It then follows from \eqref{eq:nabla A L2} that
  %\begin{equation}
  %   \| \nabla A\|_{L^2} \lesssimM (t)^2 E_s^1 (t) . \label{VAL2}
  %\end{equation} 
  Next observe
  \begin{align}
      \| \nabla A\|_{\dot{H}^s} &\lesssim \| \nabla^2 A\|_{H^{s - 1}}% + \| \nablaA\|_{L^2} 
     % \nonumber \\
      % 
      \lesssim 
      \varepsilon \|v - w\|_{H^{s - 1}} + \| \rho u \|_{H^{s -
      1}} + \| \rho A \|_{H^{s - 1}} %+ \| \nabla A\|_{L^2} 
      . \label{VAHs step1}
  \end{align}
  We have
  \begin{align}
    \varepsilon \|v - w\|_{H^{s - 1}} & \lesssim \varepsilon \| a \|_{L^{\infty}} \| \nabla
    a\|_{H^{s - 1}} \nonumber\\
    & \lesssim   M (t) E_s^1 (t),  \label{vw est}
    \end{align}
    and
    \begin{align}
    \| \rho u \|_{H^{s - 1}} %& \lesssim  \| \rho \|_{L^{\infty}} \|u\|_{H^{s- 1}} + \| u \|_{L^{\infty}} \| \rho \|_{H^{s - 1}} \nonumber\\
    & \lesssim  \| a \|_{L^{\infty}}^2 \|u\|_{H^{s - 1}} + \| u
    \|_{L^{\infty}} \| a \|_{L^{\infty}} \| a \|_{H^{s - 1}} \nonumber\\
    & \lesssim  M (t)^2 E_s^1 (t) .  \label{rhou est}
  \end{align}
 In order to estimate $\| \rho A \|_{H^{s - 1}}$ we use the Kato-Ponce inequality and observe
  %{\color{red}\begin{align*}
   % \| \rho A\|_{H^{s - 1}} & \lesssim  \sum_{|\alpha| + n (\beta)\leq s - 1, n (\beta) \geq 1} \| \partial^{\alpha} \rho\partial^{\beta} A\|_{L^2} + \sum_{|\alpha| = s - 1} \|\partial^{\alpha} \rho A\|_{L^2}\\
   % & \lesssim  \| \rho \nabla A\|_{H^{s - 2}} + \sum_{|\alpha| = s - 1}  \| \partial^{\alpha} \rho \|_{L^2} \|A\|_{L^{\infty}}.
 % \end{align*}}
 % Therefore
  \begin{align*}
    \| \rho A\|_{H^{s - 1}} & \lesssim  \| \nabla A\|_{H^{s - 2}} \| \rho
    \|_{L^{\infty}} +\| \nabla A\|_{L^{\infty}} \| \rho \|_{H^{s - 2}} + \|
    \rho \|_{H^{s - 1}} \|A\|_{L^{\infty}}\\
    & \lesssim  \| \nabla A\|_{H^{s - 2}} \|a\|_{L^{\infty}}^2 +
    \|A\|_{W^{1, \infty}} \|a\|_{L^{\infty}} \| a \|_{H^{s - 1}} \\
    &\lesssim \| \nabla A\|_{H^{s - 1}}^{\frac{s - 2}{s
    - 1}}  \| \nabla A\|_{L^2}^{\frac{1}{s - 1}} \|a\|_{L^{\infty}}^2 +
    \|A\|_{W^{1, \infty}} \|a\|_{L^{\infty}} \| a \|_{H^{s - 1}} . 
  \end{align*}
  where in the last step we used interpolation.
  %By the definition of the $H^s$ norm via the Fourier transform and Hölder's inequality, we obtain
  %\begin{equation*} \| \nabla A\|_{H^{s - 2}} \leq \|\langle \xi\rangle^{s-2} \widehat{\nabla A}^{\frac{s-2}{s-1}} \|_{{\frac{2(s-1)}{s-2}}} \|\widehat{\nabla A}^{\frac{1}{s-1}} \|_{{2(s-1)}} \lesssim \| \nabla A\|_{H^{s - 1}}^{\frac{s -   2}{s - 1}}  \| \nabla A\|_{2}^{\frac{1}{s - 1}} . \end{equation*}
 % Then
 % \begin{equation*}
  %  \| \rho A\|_{H^{s - 1}} \lesssim \| \nabla A\|_{H^{s - 1}}^{\frac{s - 2}{s   - 1}}  \| \nabla A\|_{L^2}^{\frac{1}{s - 1}} \|a\|_{L^{\infty}}^2 +  \|A\|_{W^{1, \infty}} \|a\|_{L^{\infty}} \| a \|_{H^{s - 1}} . 
  %\end{equation*}
  Using Young's inequality,
  \begin{equation*} C_1 \| \nabla A\|_{H^{s - 1}}^{\frac{s - 2}{s - 1}}  (\| \nabla A\|_{L^2}
     \|a\|_{L^{\infty}}^{2 (s - 1)})^{\frac{1}{s - 1}} \leq \frac{1}{2}
     \| \nabla A\|_{H^{s - 1}} + C_2 \| \nabla A\|_{L^2} \|a\|_{L^{\infty}}^{2
     (s - 1)} . \end{equation*}
     Then
  \begin{equation}
    \| \rho A\|_{H^{s - 1}} \lesssim \frac{1}{2}
     \| \nabla A\|_{H^{s - 1}} + C_2 \| \nabla A\|_{L^2} \|a\|_{L^{\infty}}^{2
     (s - 1)} +
    \|A\|_{W^{1, \infty}} \|a\|_{L^{\infty}} \| a \|_{H^{s - 1}} . \label{rhoA est}
  \end{equation}
  Finally, we obtain
  %Combining \eqref{VAL2}, \eqref{vw est}, \eqref{rhou est} and \eqref{rhoA est} in \eqref{VAHs step1} yields
  %\begin{align*}
    %\| \nabla A\|_{H^s} %& \lesssim %\varepsilon \|v - w\|_{H^{s - 1}} + \| \rho u \|_{H^{s - 1}} + \| \nabla A\|_{L^2}\\ &\qquad + \| \nabla A\|_{H^{s - 1}}^{\frac{s - 2}{s - 1}}  \| \nabla A\|_{L^2}^{\frac{1}{s - 1}} \|a\|_{L^{\infty}}^2 + \|A\|_{W^{1, \infty}}\|a\|_{L^{\infty}} \| a \|_{H^{s - 1}}\\
    %&\lesssim \varepsilon \|v - w\|_{H^{s - 1}} + \| \rho u \|_{H^{s - 1}}+ (1 +\|a\|_{L^{\infty}}^{2 (s - 1)}) \| \nabla A\|_{L^2} + \|A\|_{W^{1,\infty}} \|a\|_{L^{\infty}} \| a \|_{H^{s - 1}}\\
    %&\lesssim  M (t) (\|A\|_{W^{1, \infty}} + M (t)) E_s^1 (t) + M (t)^{2 s - 2} \| \nabla A\|_{L^2} .
  %\end{align*}
  %  Then we can apply  and get
  \begin{equation}
    \| \nabla A\|_{H^s} \lesssim M (t) (\|A\|_{W^{1, \infty}} + M (t)) E_s^1
    (t) + M (t)^{2 s} E_s^1 (t) . \label{VLA3}
  \end{equation}
  The claim then follows from  \eqref{VALinf}.
\end{proof}

\section{Energy estimate}\label{4}

Define $E^\mu_s(t)$,
\begin{equation}
  E_s^{\mu}(t) := \| (a, u) (t) \|_{X^s} + \mu \varepsilon \|a (t) \|_{H^s}, 
  \label{Es def}
\end{equation}
where $\mu >0$.
In this section we will prove an a priori estimate for $E^1_s(t)$. By Sobolev's inequality, for $s >
7/2$, we have
  \begin{equation}
    M (t) \lesssim E_s^{\mu}  (t) . \label{MEs est}
  \end{equation}
  where $M(t)$ is defined as
  \begin{equation}
      M (t) := 1 + \|u (t) \|_{W^{1, \infty}} + \|a (t) \|_{L^{\infty}} + \varepsilon \|a (t)
  \|_{H^1} + \varepsilon\|a (t) \|_{W^{1, \infty}} %+ \varepsilon\|a (t) \|_{W^{2,3}} 
  \label{M def}
  \end{equation}
  This will allow us to close the argument and obtain an a priori estimate for $E^{1}_s(t)$ via Gronwall's inequality.
 
%For $\mu_1, \mu_2, T > 0$, we define the following space:
%\begin{equation}
%  \mathbb{C})) \times W^{1, \infty} ([0, T] , H^s (\Omega , \mathbb{R}^3)),
%  \label{Y def}
%\end{equation}
%equipped with the following norm:
%\begin{equation} \| (a, u) \|_{Y_{\mu_1, \mu_2}^{T, s}} := \sup_{t \in [0, T]} (\| (a,
%   u) (\cdot, t) \|_{X^s} + \mu_1 \varepsilon \|a (\cdot, t) \|_{H^s} +
%   \mu_2 \| \partial_t u (\cdot, t) \|_{H^{s - 1}}), 0 < \mu_1, \mu_2 < 1.
%\end{equation}
%In Section \ref{5}, we will prove the existence of solutions in $Y_{\mu_1,
%\mu_2}^{T, s}$.
\begin{proposition}
  \label{a priori opt}Suppose $(a,u)$ is a solution to
  \eqref{Madelung-Poisson} and let $s>7/2$. Then
  \begin{equation}
    E_s^1 (t) \lesssim N (t)^{2 s + 3} E_s^1 (0) e^{C N (t)^{2 s + 3} t} .
    \label{energy a priori}
  \end{equation}
  where $N(t)= \sup_{0\leq t \leq \tau} M(\tau)$.
\end{proposition}

%This proposition allows us to prove local wellposedness and the semiclassical limit $\varepsilon \rightarrow 0$ of \eqref{Madelung-Poisson}-\eqref{wkb data}. Moreover, the second part of Theorem \ref{thm:main result blow up} is then proved by a bootstrap argument, cf. Section \ref{section blow up}.

The proof of Proposition \ref{a priori opt} is split into three steps. In the first step we estimate $\|(a,u)\|_{X^s}$ (Lemma \ref{similar est}), similarly  to Lemma 2.2 in
{\cite{2003Wigner}}. 
In the second step (Lemma \ref{energy analogy}) we estimate $\varepsilon \|a (t) \|_{H^s}$. In the third step (section \ref{sec:weaker_assumptions}) we prove Proposition \ref{a priori opt}. 

We note that  \eqref{Madelung-Poisson} conserves the $L^2$-norm: Taking the $\mathbb{C}^2$ inner product with $a$ in \eqref{Madelung-Poisson} and the real part yields
\begin{equation}
  \partial_t |a |^2 + (u-A) \cdot \nabla |a |^2 + |a |^2 \mathrm{div}
  (u-A) = { \varepsilon}\Re i  \langle  {a},\Delta a\rangle. \label{mass step1}
\end{equation}
since $\Re i \langle a,(\sigma \cdot B) a \rangle=0$. Integration by parts yields
\begin{equation}
  \partial_t \|a\|_{L^2}^2 = 0. \label{a cons}
\end{equation}
Moreover, interpolation and Hölder's inequality yield
  \begin{equation}
    \|a\|_{L^p} \lesssim \| a \|_{L^2} + \| a \|_{L^{\infty}} \lesssim M (t),
    \quad 2 \leq p \leq + \infty, \label{alp ori}
  \end{equation}
 There are two key observations concerning the estimate above. (i) 
 Integration by parts transfers the derivative from $|a|^2$ to $
  u-A$. (ii) the term $i \varepsilon \Delta
  a$ cancels.

We first state a small lemma which will be useful in the proof of the a priori estimate.
\begin{lemma}
  \label{charge cons with source} Let $u,A \in W^{1,\infty}$ and $F\in L^2$. Let $a$ be a solution of
  \begin{equation*}
    \partial_t a + (u-A) \cdot \nabla a  =  \frac{i \varepsilon}{2} \Delta
    a + F,
  \end{equation*}
  Then
  \begin{equation}
    \partial_t \|a\|_{L^2}^2 = \int \mathrm{div} (u-A) | a |^2 + 2
    \int \Re
     \langle{a},F\rangle , \label{a cons eq}
  \end{equation}
  and
  \begin{equation}
    \partial_t \|a\|_{L^2} \lesssim (M (t) + \|A\|_{W^{1,\infty}}) \|a\|_{L^2} + \| F
    \|_{L^2} . \label{a cons with source}
  \end{equation}
\end{lemma}

\begin{proof}
  Taking the $\mathbb{C}^2$ inner product with $a$, taking the real part and integrating by parts yields 
  \eqref{a cons eq}. Then \eqref{a cons with source} follows from \eqref{a cons eq}, Hölder's inequality and
  \begin{align*}
    \left| \int \mathrm{div} (u-A) | a |^2 \right| & \leq (\|
    \nabla u \|_{L^{\infty}} + \| \nabla A \|_{L^{\infty}}) \|a\|_{L^2}^2\\
    & \leq (M (t)+ \|A\|_{W^{1,\infty}}) \|a\|_{L^2}^2.
  \end{align*}
\end{proof}

\subsection{Preliminary estimates}\label{5.1}

In order to estimate $E_s^1(t)$ we need to estimate both $\|(a,u)\|_{X^s}$ (Lemma \ref{similar est}) and $\varepsilon\|a\|_{H^s}$ (Lemma \ref{energy analogy}).

\begin{lemma}
  \label{similar est}Let $(a,u)$ be a solution of \eqref{Madelung-Poisson}-\eqref{wkb data}. Then
  \begin{equation}
    \frac{\dd}{\dd t}  \|(a,u)(t)\|_{X^s} \lesssim M (t)^{2 s + 1} E_s^1 (t),
    \label{step 2}
  \end{equation}
\end{lemma}

\begin{proof}
  Apply $D^{s-1}$ to
  (\ref{Madelung-Poisson}),
  \begin{equation}
    \partial_t D^{s-1} a + (u-A) \cdot \nabla
    D^{s-1}a = \frac{i \varepsilon}{2} \Delta
    D^{s-1} a - \mathbf{R}, \label{a eq d}
  \end{equation}
  with
  \begin{align}
  \begin{split}
    \mathbf{R} &= (\mathrm{div} D^{s-1} ((u-A) a) - (u-A) \cdot
    \nabla D^{s-1} a) - \frac{1}{2} D^{s-1} (a
    \mathrm{div} (u-A)) - \frac{i}{2} D^{s-1} ((\sigma \cdot B)a)
    \\ &= F_1 + F_2 + F_3
    \end{split}
    \label{Rc def}
  \end{align}
  Applying \eqref{a cons with source} from
  Lemma \ref{charge cons with source} yields
\begin{equation*} \partial_t \| D^{s-1} a \|_{L^2} \lesssim (M (t) + \| A
     \|_{W^{1, \infty}}) \| D^{s-1} a \|_{L^2} + \| \mathbf{R}
     \|_{L^2}. \end{equation*}
  %It is easy to see that $\mathbf{R}$ does not contain $s$-th order derivative of $a$. 
  We estimate the terms in $\mathbf{R}$. The Kato-Ponce inequality yields for $F_1$:
  \begin{align*}
     \|F_1\|_2 %&=  \|\mathrm{div} D^{s-1} ((u-A) a) - (u-A) \cdot\nabla D^{s-1} a \|_{2} \\
    &\lesssim  \| a \|_{{H}^{s - 1}}   \| \nabla (u-A) \|_{L^{\infty}} +
    \| a \|_{L^{\infty}}   \| u-A \|_{\dot{H}^s},\\
    &\lesssim (M (t) + \| A \|_{W^{1, \infty}}) (\|(a,u)\|_{X^s} +\|
    \nabla A\|_{H^s}).
    \end{align*}
    For $F_2$ we have
    \begin{align*}
     \|F_2\|_2 %&= \| \partial^{\alpha}_x (a \mathrm{div} (u-A)) \|_{2} \\
    &\lesssim  \| a \|_{{H}^{s - 1}}   \| \nabla (u-A) \|_{L^{\infty}} +
    \| a \|_{L^{\infty}}   \| u-A \|_{\dot{H}^s}\\
    &\lesssim  (M (t) + \| A \|_{W^{1, \infty}}) (\|(a,u)\|_{X^s} +\| \nabla
    A\|_{H^s}).
    \end{align*}
    For $F_3$ we have
    \begin{align*}
     \|F_3\|_2 %&= \| \partial^{\alpha}_x ((\sigma \cdot B) a) \|_{2} \\
    &\lesssim  \| \nabla A \|_{L^{\infty}} \| a \|_{H^{s - 1}} + \|
    \textbf{} a\|_{L^{\infty}} \| \nabla A \|_{H^{s - 1}}\\
    &\lesssim  (M (t) + \| A \|_{W^{1, \infty}}) (\|(a,u)\|_{X^s} +\| \nabla
    A\|_{H^s}) .
  \end{align*}
  %Here, we used the fact $1 \leq |\alpha| \leq s - 1$. %In theestimate for the last term, we also used the following fact,
  %\begin{equation}
   % \| \mathbf{B} \|_{H^{s - 1}} \lesssim \| \nabla A\|_{H^{s - 1}}, \quad \|
    %\mathbf{B} \|_{L^{\infty}} \lesssim \| \nabla A\|_{L^{\infty}} .
    %\label{bold B}
  %\end{equation}
  It follows that
  \begin{equation}
    \| \mathbf{R} \|_{L^2} \lesssim (M (t) + \| A \|_{W^{1, \infty}}) (\|(a,u)\|_{X^s} +\|
    \nabla A\|_{H^s}) . \label{Rc est}
  \end{equation}
  Therefore,
  %\begin{equation*} \partial_t  \| D^{s-1} a\|_{L^2} \lesssim (M (t) + \| A
   %  \|_{W^{1, \infty}}) (\|(a,u)\|_{X^s} +\| \nabla A\|_{H^s}) . \end{equation*}
  %Therefore we obtain
  \begin{equation}
    \partial_t \|a\|_{H^{s - 1}} \lesssim (M (t) + \| A \|_{W^{1, \infty}})
    (\|(a,u)\|_{X^s} +\| \nabla A\|_{H^s}) \label{a a priori estimate} .
  \end{equation}
  The estimate for 
  $\|u\|_{H^s}$ is similar and we obtain
  \begin{equation}
    \partial_t \|u\|_{H^s} \lesssim (M (t) + \| A \|_{W^{1, \infty}}) (E_s
    (t) +\| \nabla A\|_{H^s}) + \| \nabla V\|_{H^s} . \label{u a priori estimate}
  \end{equation}
  Combining \eqref{a a priori estimate} and \eqref{u a priori estimate}
  and Lemma \ref{thm:estimates_poisson} finishes the proof. 
\end{proof}

%\subsection{Estimate of \texorpdfstring{$\varepsilon \|a\|_{H^s}$}{ε||a||Hs}}\label{5.3}

\begin{lemma}
  \label{energy analogy} Let $(a,S)$ be a solution to
  \eqref{Madelung-Poisson 0.5}-\eqref{Madelung-Poisson 0.5 2}. Then
  \begin{equation}
    \varepsilon^2  \partial_t \| D^s a\|^2_{L^2} - 2
    \varepsilon \int (D^s S) \partial_t
    \Im \langle{a} ,D^s a \rangle \dd x \lesssim M (t)^{2 s + 2}
    E_s^1 (t) . \label{similar energy}
  \end{equation}
  where $E_s^1 (t)$ is defined by
  \eqref{Es def}.
\end{lemma}

Lemma \ref{energy analogy} in combination with Lemma \ref{similar est} will
yield Proposition \ref{a priori opt}. 

\begin{proof}
Apply $D^s$ to \eqref{Madelung-Poisson},
  \begin{equation}
    \partial_t D^s a + (u-A) \cdot \nabla
    D^s a + \frac{1}{2} a D^s \mathrm{div}
    (u-A) + \mathbf{R}_a = \frac{i \varepsilon}{2} \Delta D^s a,
    \label{new approach}
  \end{equation}
  with
\begin{equation}
    \begin{split}
      \mathbf{R}_a =& (D^s ((u-A) \cdot \nabla a) - (u-A)
      \cdot \nabla D^s a) - \frac{i}{2}
      D^s ((\sigma \cdot B)a)\\
      &+  \frac{1}{2}  (D^s (a \mathrm{div} (u-A)) - a D^s \mathrm{div} (u-A)).
    \end{split} \label{Ra def}
\end{equation}
Multiplying by $\varepsilon^2$ and using Lemma \ref{charge cons with source} leads to
  \begin{align}
    \varepsilon^2\partial_t \| D^s a \|^2_{L^2} &= - 
    \varepsilon^2\int (D^s \mathrm{div} (u-A))\Re\langle D^s  {a} ,a 
    \rangle \dd x 
    + \varepsilon^2\int | D^s a |^2
    \mathrm{div} (u-A) \dd x \nonumber \\ &\qquad - \varepsilon^2\int \Re \langle D^s 
    {a }, \mathbf{R}_a  \rangle \dd x. \label{eq:intermediate 1}
  \end{align}
  On the other hand, taking the $\mathbb{C}^2$ inner product of  \eqref{new approach} with ${a}$ yields
  \begin{equation*} \langle{a}, \partial_t D^{s} a\rangle + (u-A) \cdot\langle{a}  ,
     \nabla D^{s} a\rangle + \frac{1}{2} |a |^2 D^{s}
     \text{div} (u-A) + \langle{a}, \mathbf{R}_a\rangle = \frac{i \varepsilon}{2}  \langle{a},
     \Delta D^{s} a \rangle. \end{equation*}
  Multiplying \eqref{Madelung-Poisson} by
  $D^{s}  \overline{a}$ and complex conjugation yields
  \begin{equation*} \langle \partial_t  {a}, D^{s} a\rangle  + (u-A) \cdot \langle \nabla {a},D^{s} a\rangle 
      + \frac{1}{2}  \text{div} (u-A)\langle{a},
     D^{s} a \rangle = - \frac{i \varepsilon}{2}
     \langle\Delta {a}, D^{s} a\rangle - \frac{i}{2} \langle (\sigma \cdot B){a} ,D^{s} a\rangle . \end{equation*}
  Adding the two equations yields
  \begin{equation} \partial_t  \langle {a} ,D^{s} a\rangle + (u-A) \cdot
     \nabla \langle{a}, D^{s} a\rangle + \frac{1}{2}  | a |^2
     D^{s} \text{div} (u-A) + \mathbf{R}_b = \frac{i \varepsilon}{2}
     \langle a, \Delta D^{s} a\rangle - \frac{i \varepsilon}{2}
     \langle \Delta {a} ,D^{s} a \rangle, 
     \label{eq: I_1 intermediate step} \end{equation}
  where $\mathbf{R}_b$ is defined as
  \begin{equation}
    \mathbf{R}_b := \langle{a} ,\mathbf{R}_a\rangle + \frac{1}{2} \text{div} (u-A) \langle{a} ,D^{s}
    a\rangle  + \frac{i}{2}  \langle (\sigma \cdot B){a},
    D^{s} a\rangle . \label{Rb def}
  \end{equation}
  Taking the imaginary part of \eqref{eq: I_1 intermediate step} and multiplying with $D^s S$ yields
  \begin{align*} &D^s S \partial_t \Im 
    \langle{a}, D^s a\rangle + D^s S (u-A) \cdot \nabla
     \Im 
    \langle{a}, D^s a\rangle + D^s
     S \Im \mathbf{R}_b \\ =& - \frac{\varepsilon}{2} D^s S
     \Re (\langle{a} , \Delta D^s a\rangle - \langle \Delta
     {a} ,D^s a\rangle ) . 
\end{align*}
  Integration by parts yields:
  \begin{align*}
     \int D^s S
     \Re (\langle{a} , \Delta D^s a\rangle - \langle \Delta
     {a} ,D^s a\rangle ) \dd x = \int (\Delta D^s S)
    \Re \langle{a} ,D^s a \rangle +2\nabla D^s S \cdot \Re \langle \nabla {a} ,D^s a \rangle \dd x
  \end{align*}
  Therefore, using $u=\nabla S$ and integration by parts,
  \begin{align}
        &-2\varepsilon \int D^s S \partial_t \Im 
    \langle{a}, D^s a\rangle \dd x = -2\varepsilon\int \text{div}
    ((D^s S) (u-A)) \Im \langle{a},
    D^s a\rangle \dd x + 2\varepsilon \int D^s S
    \Im \mathbf{R}_b \dd x \nonumber \\
    &\qquad  + {\varepsilon^2} \int (D^s \mathrm{div}
    u )\Re \langle{a} ,D^s a \rangle \dd x + 2\varepsilon^2
    \int D^s u \cdot \Re \langle \nabla {a},
    D^s a \rangle \dd x. \label{eq:intermediate 2}
  \end{align}
Thus adding \eqref{eq:intermediate 1} and \eqref{eq:intermediate 2} yields
\begin{align}
   \varepsilon^2 \partial_t \| D^s a\|^2_{L^2} - 2
    \varepsilon \int D^s S \partial_t
    \Im \langle{a} ,D^s a \rangle \dd x &=\varepsilon^2\int | D^s a |^2
    \mathrm{div} (u-A) \dd x - \varepsilon^2\int \Re \langle D^s 
    {a }, \mathbf{R}_a  \rangle \dd x \nonumber\\
    &\qquad -2\varepsilon\int \text{div}
    ((D^s S) (u-A)) \Im \langle{a},
    D^s a\rangle \dd x \nonumber \\&\qquad + 2\varepsilon \int D^s S
    \Im \mathbf{R}_b \dd x  \nonumber \\
    &\qquad + 2\varepsilon^2
    \int D^s u \cdot \Re \langle \nabla {a},
    D^s a \rangle \dd x. \nonumber \\
    &\qquad -2\varepsilon\int (D^s\mathrm{div}A )\Re\langle a, D^s a\rangle  \dd x \nonumber \\
    &=: F_1 + F_2 + F_3 +F_4 +F_5 + F_6 \nonumber
\end{align}
    This eliminates the $s+1$-th order derivative of $u$. Now we can estimate the RHS, which proves the claim. Starting with $F_1$,
    \begin{align*}
    \varepsilon^2\left| \int | D^s a |^2 \mathrm{div} (u-A) \dd x
    \right| &\lesssim \varepsilon^2 \|  a\|_{H^s}^2  \|u-A\|_{W^{1,
    \infty}} \lesssim {M (t)^5} E_s^1 (t)^2.  %\label{est start eps2}
    \end{align*}
For $F_2$ we have 
    \begin{align*}
        \varepsilon^2\Re \langle D^s 
    {a }, \mathbf{R}_a  \rangle \dd x &= \varepsilon^2\int \Re \langle D^s  {a}, D^s ((u-A) \cdot
    \nabla a) - (u-A) \cdot \nabla D^s a
    \rangle \dd x \\
    &+ \frac{\varepsilon^2}{2}\int  \Im \langle D^s  {a}, D^s
    ((\sigma \cdot B)a) \rangle  \dd x \\
    &+ \frac{\varepsilon^2}{2}\int \Re \langle D^s
    {a}, D^s (a \mathrm{div} (u-A)) - a D^s \mathrm{div} (u-A) \rangle \dd x \\
    &:= G_1 + G_2 + G_3
    \end{align*}
    For $G_1$ we have
    \begin{align*}
     |G_1| 
    &\lesssim \varepsilon^2(\|a\|_{H^s} \|u-A\|_{W^{1, \infty}} +\|u-A\|_{\dot{H}^s} \|
    \nabla a\|_{L^{\infty}})  \| a\|_{H^s} \nonumber\\
    &\lesssim  M (t)^{2 s + 1} E_s^1 (t)^2 .  %\label{est end eps2}
  \end{align*}
    For $G_2$ we have
    \begin{align*}
     |G_2| 
    %&\lesssim (\|a\|_{H^s} \| \mathbf{B} \|_{L^{\infty}} +\| \mathbf{B}\|_{H^s} \|a\|_{L^{\infty}})  \| \partial^{\alpha}_x a\|_{L^2} \nonumber\\
    &\lesssim \varepsilon^2(\|a\|_{H^s} \| \nabla A\|_{L^{\infty}} +\| \nabla A\|_{H^s}
    \|a\|_{L^{\infty}})  \| a\|_{H^s} \nonumber\\
    &\lesssim \varepsilon M (t)^{2 s + 1} E_s^1 (t)^2.
  \end{align*}
  And for $G_3$ we have
  \begin{align*}
    |G_3|
    &\lesssim \varepsilon^2(\|a\|_{H^s} \|u-A\|_{W^{1, \infty}} +\|u-A\|_{\dot{H}^s} \|
    \nabla a\|_{L^{\infty}})  \| a\|_{H^s} \nonumber\\
    &\lesssim  M (t)^{2 s + 1} E_s^1 (t)^2 ,
    \end{align*}
    For $F_3$ we have
   \begin{align*}
       \varepsilon \int D^s S \text{Im} \mathbf{R}_b \dd x &=   \frac{1}{2}\int D^s S \Re (\langle{a}
    D^s (\sigma \cdot {B})a\rangle - \langle \sigma \cdot{B}{a},
    D^s a\rangle) \dd x\\
    %&\qquad + \varepsilon \int D^s S \Re \langle \Delta {a},
   % D^s a\rangle \dd x \\
   &\qquad + \varepsilon \int D^s S \Im \langle a, (D^s ((u-A) \cdot \nabla a) - (u-A)
      \cdot \nabla D^s a)\rangle \dd x \\ &\qquad + \frac{\varepsilon}{2}
      \int (D^s S) \text{div} (u-A)\text{Im} \langle{a}
    D^s a)  \dd x \\
    &\qquad +\frac{\varepsilon}{2} \int (D^s S) \Im \langle a, D^s (a \mathrm{div} (u-A)) - a D^s \mathrm{div} (u-A)\rangle \dd x \\
    &= H_1 + H_2 + H_3 + H_4
   \end{align*}
    For $H_1$ we have
    \begin{align*}
    |H_1| &\lesssim \varepsilon \| u\|_{H^{s-1}}  (\|a \|_{\infty} \| D^s (\sigma \cdot{B})a)\|_{2} +\| D^s a \|_{2} \| (\sigma \cdot {B})a
    \|_{\infty})\\
   % &\lesssim  \|u\|_{H^s} \|a \|_{\infty}  (\| \mathbf{B} \|_{H^s}\|a \|_{\infty} +\|a \|_{H^s} \| \mathbf{B} \|_{\infty})\\
   % &\lesssim  \|u\|_{H^s}  \|a \|_{\infty}  (\| \nabla A\|_{H^s}
   % \|a \|_{\infty} +\|a \|_{H^s} \| \nabla A\|_{\infty})\\
    &\lesssim  \varepsilon M (t)^{2 s + 2} \|u\|_{H^s}  (\|u\|_{H^s} +\|a \|_{H^s})\\
    &\lesssim  {M (t)^{2 s + 2}} E_s^1 (t)^2.
    \end{align*}
    %For $H_2$ we have using Sobolev's inequality
    %\begin{align*}
    %|H_2| %&\lesssim \| \Delta a \|_{L^3} \| \partial^{\alpha}_x S\|_{L^6}  \| \partial^{\alpha}_x a \|_{2}\\
    %&\lesssim  \varepsilon^2\| \Delta a \|_{L^3}  \| \nabla D^s
    %S\|_{2}  \| D^s a \|_{2}\\
    %&\lesssim {M (t)}{} E_s^1 (t)^2.
    %\end{align*}
    For $H_2$ we have using the Kato-Ponce inequality
    \begin{align*}
    |  H_2|
    %&\lesssim \| \partial^{\alpha}_x S\|_{2} \|a \|_{\infty}  \|\partial^{\alpha}_x ((u-A) \cdot \nabla a) - (u-A) \cdot \nabla\partial^{\alpha}_x a \|_{2}\\
    &\lesssim \varepsilon\| D^s S\|_{2} (\| \nabla (u +
    A)\|_{\infty} \| \nabla a \|_{\dot{H}^{s - 1}} +\| \nabla a
    \|_{\infty} \|u-A\|_{\dot{H}^s})\\
    \lesssim& {M (t)^{2 s + 1}} E_s^1 (t)^2 .
  \end{align*}
    For $H_3$ we have
    \begin{align*}
    | H_3 | &\lesssim \varepsilon\|
    \text{div} (u-A)\|_{\infty} \|a \|_{\infty} \|
    \partial^{\alpha}_x a \|_{2}  \| \partial^{\alpha}_x S\|_{2}\\
    &\lesssim {M (t)^6} E_s^1 (t)^2.
  \end{align*}
    For $H_4$ we have using the Kato-Ponce inequality
    \begin{align*}
    |H_4|
     %&\lesssim  \| \partial^{\alpha}_x S\|_{2} \|a \|_{\infty}  \|\partial^{\alpha}_x (a \text{div} (u-A)) - a \partial^{\alpha}_x\text{div} (u-A)\|_{2}\\
    &\lesssim \varepsilon \|u\|_{H^s}  (\|a \|_{W^{1, \infty}} \| \text{div} (u -A)\|_{H^{s - 1}} +\| \text{div} (u-A)\|_{\infty} \|a \|_{H^s})\\
    &\lesssim M (t) \|u\|_{H^s}  ((\|u\|_{H^s} +\| \nabla
    A\|_{H^{s - 1}}) + \varepsilon \|a \|_{H^s})\\
    &\lesssim {M (t)^{2 s + 1}} E_s^1 (t)^2 .
    \end{align*}
    For $F_4$ we have
    \begin{align*}
    \varepsilon\left| \int \text{div}
    ((D^s S) (u-A)) \Im \langle{a},
    D^s a\rangle \dd x \right| %&\lesssim  \| D^s a \|_{2} \|a \|_{\infty} \| \text{div}(D^s S (u-A))\|_{2}\\
    &\lesssim \varepsilon \| D^s a \|_{2} \|a \|_{\infty} 
    \| D^s S\|_{H^1}  \|u-A\|_{W^{1, \infty}}\\
    &\lesssim  {M (t)^6} E_s^1 (t)^2.
    \end{align*}
       For $F_5$ we have
    \begin{align*}
    \varepsilon^2\left| \int D^s u \cdot \Re \langle \nabla
    {a} ,D^s a \rangle \dd x \right| &\lesssim  \varepsilon^2\|u\|_{H^s}  \|
    \nabla a \|_{\infty} \|a \|_{H^s} %\lesssim  \frac{M (t)}{\varepsilon} \|u\|_{H^s} \|a\|_{H^s}\\
    \\ &\lesssim  {M (t)} E_s^1 (t)^2.
    \end{align*}  
    For $F_6$ we have
    \begin{align*}
        \varepsilon\left|\int (D^s\mathrm{div}A )\Re\langle a, D^s a\rangle  \dd x\right| &\lesssim \varepsilon \|\nabla A\|_{H^s} \|a\|_{\infty}\|a\|_{H^s} \\
        &\lesssim M(t)^{2s+2} E_s^1(t)^2.
    \end{align*}
   This completes the proof.
\end{proof}

\subsection{Proof of Proposition \texorpdfstring{\ref{a priori
opt}}{1}}\label{sec:weaker_assumptions}

Together with Lemma \ref{similar est} and Lemma \ref{energy analogy} we now prove Proposition \ref{a priori opt}.

\begin{proof}[Proof of Proposition \ref{a priori opt}]
  Note that
  \begin{align*}
     & \varepsilon^2  \partial_t \| D^s a\|^2_{L^2} - 2
    \varepsilon  \int D^s S \partial_t
    \Im \langle{a} ,D^s a\rangle  \dd x\\
     = & \varepsilon^2  \partial_t \| D^s a\|^2_{L^2} - 2
    \varepsilon \partial_t  \int D^s
    S \Im \langle{a} ,D^s a\rangle  \dd x + 2 \varepsilon
     \int (\partial_t D^s S) \Im \langle{a} ,D^s a\rangle  \dd x.
  \end{align*}
  %We claim that
  %\begin{equation*} 2 \varepsilon  \int (\partial_t D^s   S) \Im \langle{a} ,D^s a\rangle  \dd x \lesssim M (t)^{2 s +  2} E_s^1 (t)^2 . \end{equation*}
     Observe that using the equation for $u$,
  \begin{align}
    \| \partial_t u\|_{H^{s - 1}} & \lesssim  \| (u-A) \cdot \nabla
    u\|_{H^{s - 1}} + \|u \cdot \nabla A\|_{H^{s - 1}} + \| \nabla
    {|A|^2} \|_{H^{s - 1}} + \| \nabla V\|_{H^{s - 1}} \nonumber\\
    & \lesssim   (\|u\|_{L^{\infty}} +\|A\|_{L^{\infty}})  (\|u\|_{H^s}
    +\|A\|_{\dot{H}^s}) + \| \nabla V\|_{H^{s - 1}} \nonumber\\
    & \lesssim  M (t)^{2 s + 1} E_s^1 (t) .  \label{dtu est}
  \end{align}
  Therefore
  \begin{align*}
    2 \varepsilon \int (\partial_t D^s
    S) \Im \langle{a} ,D^{s} a\rangle  \dd x %& \lesssim\varepsilon \| \partial_t D^s S\|_{L^2}  \|D^s a\|_{L^2} \|a\|_{L^{\infty}}\\
    & \lesssim \varepsilon \| \partial_t u\|_{H^{s - 1}}  \|a\|_{H^s}
    \|a\|_{L^{\infty}} \\
    &\lesssim M (t)^{2 s + 1} E_s^1 (t).
  \end{align*}
  Therefore, together with Lemma \ref{energy analogy} it follows that
  \begin{equation*} \varepsilon^2  \partial_t \| D^s a\|^2_{2} - 2
     \varepsilon \partial_t  \int
     D^s S \Im \langle{a} ,D^s a\rangle  \dd x
     \lesssim M (t)^{2 s + 2} E_s^1 (t)^2 . \end{equation*}
  Integration from $0$ to $t$ yields
    \begin{align}
     \varepsilon^2 \| D^s a \|^2_{L^2} - 2 \varepsilon
     \int D^s S \Im \langle{a} ,D^s a\rangle  \dd x& \nonumber \\
    -  \varepsilon^2 \| D^s a^{\varepsilon,0} \|^2_{L^2} + 2 \varepsilon
     \int D^s S^{\varepsilon,0} \Im \langle{a^{\varepsilon,0}} ,D^s a^{\varepsilon,0}\rangle  \dd x& \lesssim \int_0^t M (r)^{2
    s + 2} E_s^1 (r)^2 \dd r.  \label{integral lemma 8}
  \end{align}
  We have for the first term on the LHS
  \begin{align*}
    \left|2\varepsilon \int D^s S \Im \langle{a} ,D^s a\rangle  \dd x \right| & \lesssim  \varepsilon \|a\|_{L^{\infty}}  \|
    D^s S\|_{L^2}  \| D^s a\|_{L^2}\\
    & \lesssim  \varepsilon M (t) \|u\|_{H^{s - 1}} \|a\|_{H^s}\\
    & \lesssim \varepsilon M (t) \|(a,u)(t)\|_{X^s}\|a\|_{{H}^s},
  \end{align*}
  Moreover,
  \begin{align*}
    \varepsilon^2 \| D^s a^{\varepsilon,0} \|^2_{L^2} + \varepsilon \left|
     \int D^s S^{\varepsilon,0} \Im \langle{a^{\varepsilon,0}} ,D^s a^{\varepsilon,0}\rangle  \dd x \right| &\lesssim
     \varepsilon^2 \| a^{\varepsilon,0} \|^2_{H^s}+\varepsilon \|u^{\varepsilon,0} \|_{H^{s - 1}} \|a^{\varepsilon,0} \|_{H^s} \|a^{\varepsilon,0} \|_{L^{\infty}} \\
    &\lesssim M (0) E_s^1 (0)^2 .
  \end{align*}
  Returning to \eqref{integral lemma 8} we have
  \begin{equation}
    \varepsilon^2 \|a\|_{{H}^s}^2 \lesssim M (0) E_s^1 (0)^2 + \varepsilon
    M (t) \|(a,u)(t)\|_{X^s} \|a\|_{{H}^s} + \int_0^t M (r)^{2 s +
    2} E_s^1 (r)^2 \dd r. \label{aHs int}
  \end{equation}
  Since for $C > 0$
  \begin{equation*} C \varepsilon M (t) \|(a,u)(t)\|_{X^s} \|a\|_{{H}^s} \leq \frac{1}{2}
     \varepsilon^2 \|a\|_{{H}^s}^2 + \frac{1}{2} C^2 M (t)^2 \|(a,u)(t)\|_{X^s}^2, \end{equation*}
  it follows that
  \begin{equation}
    \varepsilon^2 \|a\|_{{H}^s}^2 \lesssim M (0) E_s^1 (0)^2 + M (t)^2 \|(a,u)(t)\|_{X^s} + \int_0^t M (r)^{2 s + 2} E_s^1 (r)^2 \dd r. \label{aHs est}
  \end{equation}
  Rewriting \eqref{step 2} in integral form
  \begin{equation}
    \|(a,u)(t)\|_{X^s}^2 \lesssim M (0)^{2 s + 1} E_s^1 (0)^2 + \int_0^t M (r)^{2 s + 1}
    E_s^1 (r)^2 \dd r,  \label{int form}
  \end{equation}
  and inserting \eqref{int form} into \eqref{aHs est} yields
  \begin{equation}
    \varepsilon^2 \|a\|_{{H}^s}^2 \lesssim M (t)^2 M (0)^{2 s + 1} E_s^1
    (0)^2 + M (t)^2 \int_0^t M (r)^{2 s + 1} E_s^1 (r)^2 \dd r. \label{ahs est}
  \end{equation}
  Thus,
  \begin{equation*} E_s^1 (t)^2 \lesssim N (t)^{2 s + 3} \left( E_s^1 (0)^2 + \int_0^t E_s^1
     (r)^2 \dd r \right) . \end{equation*}
  and the claim follows from Gronwall's inequality.
\end{proof}

%%%%%%%%%%%%%%%%%%%%%%%%%%%%%%%%%%%%%%%%%%%%%%%%%

\section{Existence and Uniqueness of Solutions}\label{5}

In this section we show the existence and uniqueness of local solutions for the Pauli-Poisswell-WKB
equation \eqref{Madelung-Poisson}-\eqref{wkb data}  (Proposition \ref{existence}). The local wellposedness of the Euler-Poisswell equation \eqref{Euler-Poisson} as well as the Pauli-Darwin and Euler-Darwin equations can be proved analogously.  This proves Theorem \ref{thm:main result wellposedness}. %The proof of this part is even simpler because we don't need to deal with the $H^s$ norm of $a$. 
We  use a fixed point argument in the space $Y_{\mu_1, \mu_2}^{T, s}$ %by showing that the solution map $\Psi$ is a contraction on a suitable subset 
for some subtly selected $\mu_1, \mu_2$.

Let $\mu_1,\mu_2 \in(0,1)$ and define
\begin{equation}
    E^{\mu_1,\mu_2}_s(t) := \| (a,
   u) (t) \|_{X^s} + \mu_1 \varepsilon \|a (t) \|_{H^s} +
   \mu_2 \| \partial_t u (t) \|_{H^{s - 1}}
\end{equation}
Let
\begin{equation}
  Y_{\mu_1, \mu_2}^{T, s} := L^{\infty} ([0, T] , H^s) \times (L^{\infty} ([0, T] , H^{s})\cap W^{1, \infty} ([0, T] , H^{s-1})),
  \label{Y def}
\end{equation}
equipped with the norm
\begin{align*} \| (a, u) \|_{Y_{\mu_1, \mu_2}^{T, s}} &:= \sup_{t\in[0,T]} E^{\mu_1,\mu_2}_s(t) 
\end{align*}
In order to improve readability we will often write $\|(a,u)\|_Y$. Recall that $X^s$ is defined as
\begin{equation}
    X^s := H^{s-1} \times H^s, \qquad \| (a, u) \|_{X^s} :=   
  \|a\|_{H^{s - 1}} + \|u\|_{H^s}.
\end{equation}
\begin{proposition}
  \label{existence} Under the assumptions of Theorem \ref{thm:main result wellposedness}  there exists a
  unique solution $(a^{\varepsilon}, u^{\varepsilon}) \in Y_{\mu_1, \mu_2}^{T, s}$ of the Pauli-Poisswell-WKB equation
  \eqref{Madelung-Poisson}-\eqref{wkb data} for some $T=T(Q,s) > 0$ independent of $\varepsilon$.
\end{proposition}
%The proof is contained in Lemmas \ref{bound lem} and \ref{contract lem}.
Given $(\tilde{a},\tilde{u}) \in Y_{\mu_1, \mu_2}^{T, s}$, let $(a, u)$ be the solution of the linearized equation
\begin{align}
\label{Madelung-Poisson linear}
    \partial_t a + (\tilde{u} - \tilde{A}) \cdot \nabla a + \frac{1}{2}
    a \mathrm{div} (\tilde{u} - \tilde{A}) &= \frac{i \varepsilon}{2} \Delta
    a + \frac{i}{2} (\sigma \cdot \tilde{B})a,\\
    \partial_t u + (\tilde{u} - \tilde{A}) \cdot \nabla u + u \cdot \nabla
    \tilde{A} + \nabla (\frac{| \tilde{A} |^2}{2} + \tilde{V}) &= 0, \\
    -\Delta  \tilde{V}  &=   \tilde{\rho} =  | \tilde{a}|^2, \label{rho tilde
  def}, \\
  - \Delta \tilde{A} &=  \varepsilon \tilde{w} +
  \varepsilon \tilde{v} + \tilde{\rho}  (\tilde{u} - \tilde{A}),
    \\
    (a,u) (x,0) &= (a^{\varepsilon,0},u^{\varepsilon,0}) (x). \label{eq: tilde data}
  \end{align} 
Here
  \begin{align*}
  \tilde{v} := \frac{1}{2} \nabla \times \langle \tilde{a}, \sigma 
  \tilde{a}\rangle, &&
  \tilde{w} :=\frac{i}{2}  (\langle {\tilde{a}} ,\nabla
  \tilde{a}\rangle -\langle \nabla \tilde{a}, {\tilde{a}}\rangle).
\end{align*} $(\sigma \cdot \tilde{B})$ is defined similarly. We will also denote the functionals ${E}^{\mu}_s, {E}^{\mu_1,\mu_2}_s, {M}$ and ${N}$ with a tilde when they are defined with respect to $(\tilde{a},\tilde{u})$, i.e. $\tilde{E}^{\mu}_s, \tilde{E}^{\mu_1,\mu_2}_s, \tilde{M}$ and $\tilde{N}$. Define the solution map
\begin{equation*}
  \Psi^{\varepsilon} : Y_{\mu_1, \mu_2}^{T, s} 
  \rightarrow  Y_{\mu_1, \mu_2}^{T, s},\quad
  (\tilde{a}, \tilde{u})  \mapsto  (a, u).
\end{equation*}
Let $\mathcal{B} =\mathcal{B}^{T,s}_{\mu_1, \mu_2,  L}$ be the ball of radius $L$ in $Y_{\mu_1, \mu_2}^{T, s}$:
\begin{equation*} \mathcal{B}^{T,s}_{\mu_1, \mu_2,  L} := \left\{ (a, u) : \| (a, u)
   \|_{Y_{\mu_1, \mu_2}^{T, s}} \leq L \right\} . \end{equation*}

\begin{lemma}
  \label{bound lem}Let $(a^{\varepsilon,0},u^{\varepsilon,0})$ satisfy the assumptions of Theorem \ref{thm:main result wellposedness}. Then there exists $L = L(Q,s) > 0$ independent of $\varepsilon$ and $\mu_0 > 0$ such that for all $ \mu_1, \mu_2 \in(0 ,\mu_0)$ there exists $T_0 = T_0 (\mu_1, \mu_2, L, s) > 0$ independent of $\varepsilon$ such that for all $T\in [0,T_0)$
  \begin{equation*}  \Psi^{\varepsilon} (\mathcal{B}^{s,T}_{\mu_1, \mu_2, L}) \subset \mathcal{B}^{T,s}_{\mu_1, \mu_2,  L} . \end{equation*}

%{\color{red} Rewrite this Lemma, maybe part of proof of Proposition}
%\begin{lemma}
%  \label{contract lem}Under the assumptions of Theorem \ref{thm:main result wellposedness} and for any $L > 0$, there exist $\mu_{10} = \mu_{10} (L, s) > 0$ and $\mu_{20} = \mu_{20} (\mu_1, L, s) > 0$ such that for all $ \mu_1 \in (0, \mu_{10})$ and $  \mu_2 \in (0,\mu_{20})$ there exists $T_0 = T_0 (\mu_1, \mu_2, L, s) > 0$ independent of $\varepsilon$ such that for all $T \in [0, T_0)$ and $(\tilde{a}^1, \tilde{u}^1), (\tilde{a}^2  \tilde{u}^2) \in \mathcal{B}^{T,{s + 1}}_{\mu_1, \mu_2,  L}$ we have
  %\begin{equation}
   % \| \Psi^{\varepsilon} (\tilde{a}^1, \tilde{u}^1) - \Psi^{\varepsilon} (\tilde{a}^2, \tilde{u}^2) \|_{Y_{\mu_1, \mu_2}^{T, s}} \leq \tau \| (\tilde{a}^1, \tilde{u}^1) - (\tilde{a}^2, \tilde{u}^2)  \|_{Y_{\mu_1, \mu_2}^{T, s}} , \label{contract}
 % \end{equation}
%  for some constant $\tau \in (0, 1)$.
%\end{lemma}

\begin{proof}%[Proof of Lemma \ref{bound lem}] 
Let $(\tilde{a},\tilde{u})$ such that $\|(\tilde{a},\tilde{u})\|_Y \leq L$. Apply $D^s$ to \eqref{Madelung-Poisson linear},
  \begin{equation*} \partial_t D^s a + (\tilde{u} - \tilde{A})
     \cdot \nabla D^s a = \frac{i \varepsilon}{2} \Delta
     D^s a - \tilde{\mathbf{R}}_c, \end{equation*}
with
  \begin{equation*} \tilde{\mathbf{R}}_c = (\mathrm{div} D^s ((\tilde{u} -
     \tilde{A}) a) - (\tilde{u} - \tilde{A}) \cdot \nabla
     D^s a) - \frac{1}{2} D^s (a
     \mathrm{div} (\tilde{u} - \tilde{A})) - \frac{i}{2} D^s
     ( (\sigma\cdot\tilde{B})a ) . \end{equation*}
     
\emph{Step 1: Estimate for $\|(a,u)\|_{X^s}$.} Similarly to \eqref{a a priori estimate} and \eqref{u a priori estimate}, applying Lemma \ref{charge cons with source} and Sobolev's inequality \eqref{MEs est} yields
  \begin{align*}
    \partial_t \|a\|_{H^{s - 1}} &\lesssim (\tilde{M} (t)
    + \| \tilde{A} \|_{W^{1, \infty}}) \| (a, u) \|_{X^s} + M(t) (\|
    (\tilde{a}, \tilde{u}) \|_{X^s} +\| \nabla \tilde{A} \|_{H^s})\\
    &\lesssim (\tilde{E}_s^1  (t) +\| \nabla \tilde{A}
    \|_{H^s}) \| (a, u) \|_{X^s} .
  \end{align*}
and
  \begin{align*}
    \partial_t \|u\|_{H^s}  &\lesssim  (\| \nabla \tilde{A} \|_{H^s} + \|
    \tilde{u} \|_{H^s}) \|u\|_{L^{\infty}} + \|u\|_{H^s} (\| \tilde{u}
    \|_{W^{1, \infty}} + \| \tilde{A} \|_{W^{1, \infty}}) \\ &\qquad + \| \nabla
    \tilde{A} \|_{H^s}^2 + \| \nabla \tilde{V} \|_{H^s}\\
    &\lesssim (\tilde{E}_s^1
    (t) + \|\nabla \tilde{A} \|_{H^s}) \| (a, u) \|_{X^s} + \| \nabla \tilde{A} \|_{H^s}^2 + \| \nabla
    \tilde{V} \|_{H^s}.
  \end{align*}
From Lemma \ref{thm:estimates_poisson} and \eqref{MEs est} it follows that
  \begin{align*}
    \| \nabla \tilde{V} \|_{H^s} &\lesssim  \tilde{M} (t) \|
    \tilde{a} \|_{H^{s - 1}}
    \lesssim \| (\tilde{a}, \tilde{u}) \|_{Y}^2,\\
    \| \nabla \tilde{A} \|_{H^s} &\lesssim \tilde{M} (t)^{2
    s} \tilde{E}_s^1 (t)
    \lesssim \frac{\| (\tilde{a}, \tilde{u}) \|_{Y}^{2 s + 1}}{\mu_1},\\
    \| \tilde{A} \|_{W^{1,\infty}} &\lesssim \tilde{M} (t)^5
    \lesssim \| (\tilde{a}, \tilde{u}) \|_{Y}^5
  \end{align*}
 Here we used
  \begin{equation*}
    \tilde{E}_s^1 (t) \lesssim \frac{\| (\tilde{a}, \tilde{u})
    \|_{Y}}{\mu_1}, \label{EsY est}
  \end{equation*}
Therefore,
%  \begin{align*}
 %   \partial_t \| (a, u) \|_{X^s} \lesssim  \frac{\| (\tilde{a}, \tilde{u}) \|_{Y}^{2 s + 1} + 1}{\mu_1}  \| (a, u) \|_{X^s} + 
%    \frac{\| (\tilde{a}, \tilde{u}) \|_{Y}^{4 s + 2} +
%    1}{\mu_1^2} .
%  \end{align*}
%  Now multiply both sides with $\| (a, u) \|_{X^s}$
  \begin{align*}
    \partial_t \| (a, u) \|_{X^s}^2 &\lesssim  \frac{\| (\tilde{a},
    \tilde{u}) \|_{Y}^{2 s + 1} + 1}{\mu_1}  \|
    (a, u) \|^2_{X^s} +  \frac{\| (\tilde{a},
    \tilde{u}) \|_{Y}^{4 s + 2} + 1}{\mu_1^2} \| (a, u) \|_{X^s} %\\
    %&\lesssim   \frac{\| (\tilde{a}, \tilde{u}) \|_{Y}^{2 s + 1} + 1}{\mu_1}  \| (a, u) \|^2_{X^s} + 
    %\frac{\| (\tilde{a}, \tilde{u}) \|_{Y}^{4 s + 2} +1}{\mu_1^4} \| (a, u) \|_{X^s} .
  \end{align*}
  Written in integral form using the assumptions on the initial data in Theorem \ref{thm:main result wellposedness}, 
  %{\color{red}\begin{equation}
  %  \| (a, u) \|_{X^s}^2 \lesssim Q^2 +  \frac{\| (\tilde{a}, \tilde{u})  \|_{Y}^{2 s + 1} + 1}{\mu_1}  \int_0^t \| (a,  u) (r) \|^2_{X^s} \dd r + t  \frac{\| (\tilde{a}, \tilde{u}) \|_{Y}^{8 s + 4} + 1}{\mu_1^4}  . \label{ext pt1}
  %\end{equation}}
  \begin{equation}
    \| (a, u) \|_{X^s}^2 \lesssim Q^2 +  \frac{\| (\tilde{a}, \tilde{u})
    \|_{Y}^{2 s + 1} + 1}{\mu_1}  \int_0^t \| (a,
    u) (r) \|^2_{X^s} \dd r +  \frac{\| (\tilde{a}, \tilde{u})
    \|_{Y}^{4 s + 2} + 1}{\mu_1^2} \int_0^t \| (a,
    u) (r) \|_{X^s} \dd r . \label{ext pt1}
  \end{equation}
  \emph{Step 2:} In a similar fashion to the proof of Lemma \ref{energy analogy} we have
  %\begin{align*}
   % &\varepsilon^2  \partial_t \| D^s a\|^2_{L^2} + \varepsilon^2   \int (D^s \mathrm{div} \tilde{u}) \Re\langle D^s  {a}, a  \rangle  \dd x\\
  %  &\qquad \qquad \lesssim  \left( \| (\tilde{a}, \tilde{u}) \|_{Y}^{2 s + 1} + 1 \right) (\varepsilon \| a \|_{H^s} + \| (a, u)   \|_{X^s})^2\\
  %  &\qquad \qquad \lesssim   \frac{\| (\tilde{a}, \tilde{u}) \|_{Y}^{2 s + 1} + 1}{\mu_1^2}  E_s^{\mu_1} (t)^2,
  %  \end{align*}
 %   and
 %   \begin{align*}
 %   & - 2 \varepsilon \int D^s   \tilde{S} \partial_t \Im \langle {a}, D^s a \rangle \dd x -  \varepsilon^2   \int (D^s \mathrm{div} \tilde{u}) \Re \langle {a}, D^s a \rangle \dd x\\
   % &\qquad \qquad \lesssim \left( \| (\tilde{a}, \tilde{u}) \|_{Y}^{2 s + 2} + 1 \right) (\varepsilon \| a \|_{H^s} + \| (a, u)  \|_{X^s})^2\\
  %  &\qquad \qquad \lesssim   \frac{\| (\tilde{a}, \tilde{u}) \|_{Y}^{2 s + 1} + 1}{\mu_1^2}  E_s^{\mu_1}  (t)^2 .
 % \end{align*}
  %Adding both estimates leads to
  \begin{equation*}
    \varepsilon^2  \partial_t \| D^s a\|^2_{L^2} - 2
    \varepsilon  \int D^s  \tilde{S}
    \partial_t \Im \langle {a}, D^s a \rangle \dd x \lesssim
     \frac{\| (\tilde{a}, \tilde{u}) \|_{Y}^{2 s +
    2} + 1}{\mu_1^2}  E_s^{\mu_1} (t)^2,
  \end{equation*}
  where we used
  \begin{equation*} E_s^1 (t) \lesssim \frac{E_s^{\mu_1} (t)}{\mu_1}. 
    \end{equation*}
  and which is the analogue of \eqref{similar energy}. Now observe
  \begin{equation*} \partial_t \int D^s
     \tilde{S} \Im\langle {a}, D^s a \rangle \dd x =  \int D^s  \tilde{S} \partial_t \Im \langle {a}, D^s a \rangle \dd x +  \int
     (\partial_t D^s \tilde{S}) \Im \langle {a}, D^s a \rangle \dd x \end{equation*}
  The second term on the RHS can be estimated as follows,
  \begin{align*}
    \varepsilon\left|  \int (\partial_t D^s
    \tilde{S}) \Im \langle {a}, D^s a \rangle \dd x \right| %&\lesssim  \varepsilon \| \partial_t D^s  \tilde{S}\|_{L^2}  \| D^s a\|_{L^2} \|a\|_{L^{\infty}}\\
    &\lesssim  \varepsilon \| \partial_t  \tilde{u} \|_{H^{s - 1}} 
    \|a\|_{H^s} \|a\|_{L^{\infty}}\\
    &\lesssim  \frac{1}{\mu_1 \mu_2}\| (\tilde{a}, \tilde{u}) \|_{Y}
    E_s^{\mu_1} (t)^2,
  \end{align*}
  Therefore,
  \begin{equation*} \varepsilon^2  \partial_t \| D^s a\|^2_{L^2}
     - 2 \varepsilon \partial_t \int
     D^s \tilde{S} \Im \langle{a} ,D^s
     a\rangle \dd x \lesssim \frac{\| (\tilde{a}, \tilde{u}) \|_{Y}^{2 s + 2} + 1 }{(\mu_1 \wedge \mu_2) \mu_1} E_s^{\mu_1}
     (t)^2 . \end{equation*}
  Integrating from $0$ to $t$ and estimating the terms in a similar fashion to \eqref{aHs int} yields
  %\begin{align*}
   % & \varepsilon^2 \| D^s a \|^2_{L^2} - 2 \varepsilon \int D^s  \tilde{S} \Im \langle{a} ,D^s a\rangle \dd x - \varepsilon^2 \| D^s a^{\varepsilon,0} \|^2_{L^2}\\
   % + & 2 \varepsilon  \int D^s \tilde{S} (0) \Im  \langle{a}^{\varepsilon,0} ,D^s a^{\varepsilon,0}\rangle \dd x \\ &\qquad \qquad \qquad \lesssim \frac{\| (\tilde{a}, \tilde{u}) \|_{Y}^{2 s + 2} + 1}{(\mu_1 \wedge \mu_2) \mu_1} \int_0^t E_s^{\mu_1} (r)^2 \dd r.
  %\end{align*}
 % Similarly to \eqref{aHs int},
  \begin{align*}
    \varepsilon^2 \|a\|^2_{H^s} & \lesssim  (\varepsilon \| (a^{\varepsilon,0}, u^{\varepsilon,0})
    \|_{X^s} \| (\tilde{a}, \tilde{u}) (0) \|_{X^s} \| a^{\varepsilon,0} \|_{H^s} +
    \varepsilon^2 \|a^{\varepsilon,0} \|^2_{H^s}) \nonumber\\
    &\qquad +  \varepsilon \| a \|_{H^s} \| (\tilde{a}, \tilde{u}) \|_{Y} \| (a, u) \|_{X^s} + \frac{\| (\tilde{a}, \tilde{u})
    \|_{Y}^{2 s + 2} + 1}{(\mu_1 \wedge \mu_2) \mu_1}
     \int_0^t E_s^{\mu_1}(r)^2 \dd r. \end{align*}
  Due to the assumptions of Theorem \ref{thm:main result wellposedness} on the initial data,
  \begin{align}
    \varepsilon^2 \|a\|^2_{H^s} & \lesssim  Q^2 (\| (\tilde{a}, \tilde{u})
    (0) \|_{X^s} + 1) + \varepsilon \| a \|_{H^s} \| (\tilde{a},
    \tilde{u}) \|_{Y} \| (a, u) \|_{X^s} \nonumber\\
    &\qquad +   \frac{\| (\tilde{a}, \tilde{u}) \|_{Y}^{2 s + 2} + 1}{(\mu_1 \wedge \mu_2) \mu_1}  \int_0^t E_s^{\mu_1,\mu_2}
    (r)^2 \dd r.  \label{ext pt2}
  \end{align}
  
\emph{Step 3: Estimate for $\|\partial_t u\|_{H^{s-1}}$.}
  For $\partial_t u$, a very similar estimate to \eqref{dtu est} yields
  \begin{align}
    \| \partial_t u \|_{H^{s - 1}} %&\lesssim (\| \tilde{u} \|_{L^{\infty}}+\| \tilde{A} \|_{L^{\infty}})  (\|u\|_{H^s} +\| \tilde{A} \|_{\dot{H}^s})
    %+ \|u\|_{L^{\infty}} (\| \tilde{u} \|_{H^s} +\| \tilde{A} \|_{\dot{H}^s}) + \| \nabla \tilde{V} \|_{H^{s - 1}} \nonumber\\
   % &\lesssim (\| \tilde{u} \|_{H^s} +\| \tilde{A} \|_{\dot{H}^s} +\|\tilde{A} \|_{L^{\infty}})  (\|u\|_{H^s} +\| \tilde{A} \|_{\dot{H}^s}) +\| (\tilde{a}, \tilde{u}) \|_{Y}^2 \nonumber\\
    &\lesssim \left( \| (\tilde{a}, \tilde{u}) \|_{Y}^{2 s + 1} + 1 \right) \left( \| (\tilde{a}, \tilde{u}) \|_{Y}^{2 s + 1} + 1 + \| (a, u) \|_{X^s} \right) .  \label{ext
    pt3}
  \end{align}
  
\emph{Step 4: Estimate for $\|(a,u)\|_{Y}$.} %Recall that $\|(a,u)\|_Y = \sup_t E^{\mu_1,\mu_2}_s(t)$.
  Using \eqref{ext pt1}, \eqref{ext pt2}, \eqref{ext pt3} and
  \begin{align*}
      \frac{\| (\tilde{a}, \tilde{u}) \|_{Y}^{2 s + 1} + 1}{\mu_1} + \frac{\mu_1 \left( \| (\tilde{a}, \tilde{u})
    \|_{Y}^{2 s + 2} + 1 \right)}{(\mu_1 \wedge \mu_2)} \lesssim \frac{\| (\tilde{a}, \tilde{u}) \|_{Y}^{2 s + 2} + 1}{\mu_1 \wedge \mu_2},
  \end{align*}
  %\begin{align*}
   % E_s^{\mu_1, \mu_2}  (t)^2 \lesssim&  \| (a, u) \|_{X^s}^2 + \mu_1^2 \varepsilon^2 \|a\|^2_{H^s} + \mu_2^2 \| \partial_t u \|_{H^{s - 1}}^2\\
   % \lesssim&  Q^2 (\mu_1^2 \| (\tilde{a}, \tilde{u}) (0) \|_{X^s} + 1) + t \left( \frac{\| (\tilde{a}, \tilde{u}) \|_{Y}^{8 s + 4} + 1}{\mu_1^4} \right)\\
   %  &\qquad + \left( \frac{\| (\tilde{a}, \tilde{u}) \|_{Y}^{2 s + 1} + 1}{\mu_1} + \frac{\mu_1 \left( \| (\tilde{a}, \tilde{u})\|_{Y}^{2 s + 2} + 1 \right)}{(\mu_1 \wedge \mu_2)} \right) \int_0^t E_s^{\mu_1,\mu_2}  (r)^2 \dd r\\
   %  &\qquad + \mu_1^2 \varepsilon \| a \|_{H^s} \| (\tilde{a}, \tilde{u})  \|_{Y} \| (a, u) \|_{X^s}\\
   %  &\qquad + \mu_2^2 \left( \| (\tilde{a}, \tilde{u}) \|_{Y}^{4 s + 2} + 1 \right) \left( \| (\tilde{a}, \tilde{u}) \|_{Y}^{4 s + 2} + 1 + \| (a, u) \|_{X^s}^2 \right).
 % \end{align*}
  yields
  \begin{align*}
    E_s^{\mu_1, \mu_2} (t)^2  \lesssim&  Q^2 (\mu_1^2 \| (\tilde{a}, \tilde{u}) (0) \|_{X^s}
    + 1) + t \left( \frac{\| (\tilde{a}, \tilde{u}) \|_{Y}^{8 s + 4} + 1}{\mu_1^4} \right)\\
     &\qquad + \frac{\| (\tilde{a}, \tilde{u}) \|_{Y}^{2 s + 2} + 1}{\mu_1 \wedge \mu_2} \int_0^t E_s^{\mu_1,\mu_2} (r)^2 \dd r + \mu_1^2 \varepsilon \| a \|_{H^s} \| (\tilde{a}, \tilde{u})
    \|_{Y} \| (a, u) \|_{X^s}\\
     &\qquad + \mu_2^2 \left( \| (\tilde{a}, \tilde{u}) \|_{Y}^{4 s + 2} + 1 \right) \left( \| (\tilde{a}, \tilde{u}) \|_{Y}^{4 s + 2} + 1 + \| (a, u) \|_{X^s}^2 \right).
  \end{align*}
  Using $\| (\tilde{a}, \tilde{u}) \|_{Y}
     \leq L$, 
  \begin{align*}
    E_s^{\mu_1,\mu_2} (a, u) (t)^2 &\lesssim_s  Q^2 (\mu_1^2 L + 1) + \frac{t (L^{8 s + 4} + 1)}{\mu_1^4} + \left( \frac{L^{2 s + 2} + 1}{\mu_1 \wedge
    \mu_2} \right) \int_0^t E_s^{\mu_1,\mu_2} (r)^2 \dd r\\
    &\qquad +  \mu_2^2 (L^{4 s + 2} + 1)^2   + (\mu_1 L + \mu_2^2
    (L^{4 s + 2} + 1)) E_s^{\mu_1,\mu_2}
    (r)^2 ,
  \end{align*} 
  for some constant $C_s$. If we choose $L$ large enough and
  $\mu_1, \mu_2$ small enough such that
  \begin{equation*}
    C_s (\mu_1 L + \mu_2^2 (L^{4 s + 2} + 1))  \leq  \frac{1}{2},
\end{equation*}
and
\begin{equation*}
    C_s (Q^2 (\mu_1^2 L + 1) + \mu_2^2 (L^{4 s + 2} + 1)^2)  \leq 
    \frac{L^2}{4},
  \end{equation*}
  it follows that 
  \begin{equation*} E_s^{\mu_1,\mu_2}  (t)^2 \leq \frac{L^2}{2} + 2 C_s \left(
     \left( \frac{L^{2 s + 2} + 1}{\mu_1 \wedge \mu_2} \right) \int_0^t
     E_s^{\mu_1,\mu_2} (r)^2 \dd r + \frac{t (L^{8 s + 4} + 1)}{\mu_1^4} \right)
     . \end{equation*}
  %By Gronwall's inequality,
  %\begin{equation*} E_s^{\mu_1,\mu_2} (t)^2 \leq \left( \frac{L^2}{2} + \frac{2 C  (s) (L^{8 s + 4} + 1)}{\mu^4} t \right) \exp \left( 2 C_s \left(  \frac{L^{2 s + 2} + 1}{\mu_1 \wedge \mu_2} \right) t \right) . \end{equation*}
  Gronwall's inequality implies that there is a $T>0$ such that
  \begin{equation*} \|(a,u)\|^2_Y = \sup_{t\in [0,T]} E_s^{\mu_1,\mu_2} (t)^2
     \leq L^2 . 
    \end{equation*}
    This proves the lemma.
\end{proof}
\end{lemma}

\begin{proof}[Proof of Lemma \ref{existence}] Consider the sequence $(a^{n},u^{n})$, defined by
  \begin{align*}
      (a^{n+1},u^{n+1}) := \Psi^{\varepsilon} (a^{n},u^{n}), && ({a}^{0},u^{0}) \in X^{s}
  \end{align*}
  %We estimate the difference of two solutions $({a}^1,{u}^1)$ and $({a}^2,{u}^2)$ of the linearized system, i.e. 
  %\begin{equation*} (a^k, u^k) = \Psi^{\varepsilon} (\tilde{a}^k, \tilde{u}^k),  \quad k  = 1, 2, \end{equation*}
  %By definition, $(a^k, u^k)$, $k=1,2$,satisfy
  %\begin{align}
   % \label{Madelung-Poisson linear j} \begin{split}
    %  \partial_t a^k + (\tilde{A}^k + \tilde{u}^k) \cdot \nabla a^k + \frac{1}{2} a^k \mathrm{div} (\tilde{u}^k + \tilde{A}^k) &= \frac{i\varepsilon}{2} \Delta a^k + \frac{i}{2}  (\sigma \cdot \tilde{B}^k) a^k,\\
    %  \partial_t u^k + (\tilde{A}^k + \tilde{u}^k) \cdot \nabla u^k + u^k  \cdot \nabla \tilde{A}^k + \nabla (\frac{| \tilde{A}^k |^2}{2} + \tilde{V}^k) &= 0,\\
   %   (a^k,u^k)(x,0) &= (a^{\varepsilon,0}, u^{\varepsilon,0}) (x), 
 %   \end{split} 
 % \end{align}
  %where $\tilde{A}^k, \tilde{B}^k$ and $\tilde{V}^k$ are defined as above with every variable added with a superscript $k$. Define
  We will show that the sequence is Cauchy in $Y^{T,0}_{\mu_1,\mu_2}$. The claim then follows from the a priori estimate Proposition \ref{a priori opt} and interpolation with $Y^{T,s}_{\mu_1,\mu_2}$ similar to section 3.1.2 in \cite{erdougan2016dispersive}. 
  Denote
  \begin{equation*} 
  \delta a := a^{(n)} - a^{(m)} . \end{equation*}
  The quantities $(\delta u, \delta A, \delta V, \delta B, \delta \rho)$ are similarly defined. Moreover, denote
  \begin{align*}
      \tilde{a}^n = a^{(n-1)}, && \delta \tilde{a} = \tilde{a}^n-\tilde{a}^m
  \end{align*}
  The quantities
  $(\delta \tilde{a}, \delta \tilde{u}, \delta \tilde{A}, \delta \tilde{V},
  \delta \tilde{B}, \delta \tilde{\rho})$ are similarly
  defined. Then $(\delta a,\delta u)$ solves the following system: 
\begin{align}
      \nonumber \partial_t \delta a + (\delta \tilde{u} - \delta \tilde{A}) \cdot
      \nabla a^1 + (\tilde{u}^m - \tilde{A}^m) \cdot \nabla \delta a& \\ +
      \frac{1}{2} a^n \mathrm{div} (\delta \tilde{u} - \delta \tilde{A}) +
      \frac{1}{2} \delta a \mathrm{div} (u^m - A^m)
      &= \frac{i}{2} (\sigma \cdot \delta \tilde{B})a^n +
      \frac{i}{2}(\sigma \cdot \tilde{B}^{m}) \delta a +
      \frac{i \varepsilon}{2} \Delta \delta a,\label{eq:existence a} \\
      \partial_t \delta u + (\delta \tilde{u} - \delta \tilde{A}) \cdot \nabla
      u^1 + (\tilde{u}^m - \tilde{A}^m) \cdot \nabla \delta u& \nonumber \\ + \delta u
      \cdot \nabla \tilde{A}^n + u^2 \cdot \nabla \delta \tilde{A} +
      \nabla (\frac{\delta \tilde{A} \cdot (\tilde{A}^n + \tilde{A}^m)}{2} +
      \delta \tilde{V}) &= 0,\label{eq:existence u}\\
      - \Delta \delta \tilde{V} &= \delta \tilde{\rho},  \label{V tilde in
    sub}\\
    - \Delta \tilde{A} &= \varepsilon \delta
    \tilde{w} - \varepsilon \delta \tilde{v} + (\delta \tilde{\rho}) \tilde{u}^n \nonumber \\&\qquad - \tilde{\rho}^m \delta \tilde{u} + (\delta \tilde{\rho}) \tilde{A}^m - \tilde{\rho}^n \delta \tilde{A} .  \label{A tilde in sub ori}
  \\
      (\delta a,\delta u) (x,0) &= (0,0).
 \label{sub system tilde}
\end{align}
%Again we split the proof into multiple steps in order to estimate $(\delta a,\delta u)$ in $Y$, i.e. we estimate $\|(\delta a, \delta u)\|_{X^{s}}$, $\mu_1\varepsilon\| \delta a\|_{H^{s}}$ and $\mu_2\|\partial_t \delta u\|_{H^{s-1}}$. The constants in this section all depend implicitly on $\mu_1, \mu_2$ and $L$.\\

%We show that the sequence $(\delta a, \delta u)$ is Cauchy in $Y^{T,0}_{\mu_1,\mu_2}$. 
 We start with estimating $\delta a$ in $H^{s-1}$. Apply $D^{s-1}$ to \eqref{eq:existence a},

\begin{equation*}
    \partial_t D^{s-1} \delta a + (\tilde{u}^m - \tilde{A}^m)
    \cdot \nabla D^{s-1} \delta a + \mathbf{R}_{d} = \frac{i\varepsilon}{2} \Delta D^{s-1} \delta a,
  \end{equation*}
  where $\mathbf{R}_{d}$ contains the lower order terms. By Lemma \ref{charge cons with source} 
  \begin{align*}
    \partial_t \| D^{s-1} \delta a \|_{L^2} &\lesssim (M
    (\tilde{a}^m, \tilde{u}^m) (t) + \| \tilde{A}^m \|_{W^{1, \infty}}) \|
    D^{s-1} \delta a \|_{L^2} + \| \mathbf{R}_{d} \|_{L^2}\\ &\lesssim \| D^{s-1} \delta a \|_{L^2} +
    \| \mathbf{R}_{d} \|_{L^2} .
  \end{align*}
  In the second step, we used $(\tilde{a}^m, \tilde{u}^m) \in \mathcal{B}^{T,s}_{
  \mu_1, \mu_2, L}$. $\|\mathbf{R}_d\|_{L^2}^2$ can be easily estimated analogously to the proof above and it follows similarly to \eqref{a a priori
  estimate} that
  \begin{equation*} \partial_t \| \delta a\|_{H^{s - 1}} \lesssim \|
     \delta a\|_{H^{s - 1}} + \| \delta \tilde{u} \|_{H^{s}} + \| \nabla \delta
     \tilde{A} \|_{H^{s - 1}} . 
     \end{equation*}
     For $\delta \tilde{V}$ and $\delta \tilde{A}$, similar to \eqref{nabla V Hs bound},\eqref{VAL2} and \eqref{VLA3} we can prove
  that
  \begin{align*}
    \| \nabla \delta \tilde{V}\|_{H^s} &\lesssim \| \delta \tilde{a}\|_{H^{s
    - 1}} . %\label{delta V tilde est} 
    \\
    \| \nabla \delta \tilde{A}\|_{H^{s - 1}} %&\lesssim  \| \delta\tilde{u}\|_{H^s} + \| \delta \tilde{a}\|_{H^{s - 1}} + \varepsilon \| \delta \tilde{a}\|_{H^s} + \| \nabla \delta \tilde{A}\|_{L^2}\\
    &\lesssim  \| \delta
    \tilde{u}\|_{H^s} + \| \delta \tilde{a}\|_{H^{s - 1}} + \varepsilon \| \delta \tilde{a}\|_{H^s}.
  \end{align*}
     Now we estimate $\delta u$ in $H^s$. Apply $D^s$ to \eqref{eq:existence u},
  \begin{align*}
      \partial_t D^s \delta u + D^s((\delta \tilde{u} - \delta \tilde{A} )\nabla u^1)+\mathbf{R}_e = 0
  \end{align*}
  where $\mathbf{R}_e$ contains all the remaining terms, similar to $\mathbf{R}_{d}$ above. Since these are of lower order and the estimates are similar to the ones above, we omit the details. The second term yields, after using Kato-Ponce,
  \begin{equation*}
      \|D^s((\delta \tilde{u} - \delta \tilde{A} )\nabla u^n)\|_{L^2} \lesssim%\|\delta \tilde{u} - \delta \tilde{A}\|_{H^{s}} \|\nabla u^n\|_{L^\infty} +
      \|\delta \tilde{u} - \delta \tilde{A}\|_{L^\infty} \|\nabla u^n\|_{H^{s}}.
  \end{equation*}
  Take $s=0$. By the a priori estimate, for $s>7/2$, we have that $\|u^n\|_{H^1}$ is bounded. It follows that
  \begin{equation*} \partial_t \| (\delta a, \delta u) \|_{X^s} \lesssim
     \| (\delta a, \delta u) \|_{X^s} + \| (\delta \tilde{a}, \delta
     \tilde{u}) \|_{Y} . \end{equation*}
  The integral form of this estimate is
  \begin{equation}
    \| (\delta a, \delta u) \|_{X^s} \lesssim \int_0^t \|
    (\delta a, \delta u) (r) \|_{X^s} \dd r + t \| (\delta \tilde{a},
    \delta \tilde{u}) \|_{Y} . \label{delta Xs est}
  \end{equation}
  
  In a similar fashion to the proof of Lemma \ref{energy analogy}
  %Similarly to \eqref{eps2},
  %\begin{equation*} \varepsilon^2  \partial_t \| D^{s} \delta a\|^2_{L^2} +   \varepsilon^2   \int D^{s} \mathrm{div} \tilde{u}^2 \Re \langle D^{s} {\delta a}, \delta a\rangle     \dd x \lesssim \| (\delta a, \delta u)   \|_{X^s}^2 + \| (\delta \tilde{a}, \delta \tilde{u}) \|_{Y}^2 . \end{equation*}
  %Similarly to \eqref{eps1},
  %\begin{align*}  & - 2 \varepsilon  \int D^{s}   \tilde{S} \partial_t \Im \langle{\delta a} ,D^{s} \delta a\rangle \dd x - \varepsilon^2   \int D^{s} \mathrm{div} \tilde{u}^2 \Re \langle D^{s} {\delta a}, \delta a\rangle  \dd x\\ &\lesssim  \| (\delta a, \delta u) \|_{X^s}^2 + \|  (\delta \tilde{a}, \delta \tilde{u}) \|_{Y}^2 .
 % \end{align*}
  %Then similarly to \eqref{similar energy}, adding the two estimates above yields
  \begin{equation*} \varepsilon^2  \partial_t \| D^{s} \delta a\|^2_{L^2} - 2
     \varepsilon  \int D^{s}  \tilde{S}
     \partial_t  \langle{\delta a} ,D^{s}
    \delta a\rangle \dd x
     \lesssim% 
     \| (\delta a, \delta u) \|_{X^s}^2 + \|
     (\delta \tilde{a}, \delta \tilde{u}) \|_{Y}^2 . \end{equation*}
  It follows that
  \begin{equation*} \varepsilon^2  \partial_t \| D^{s} \delta a\|^2_{L^2} - 2
     \varepsilon \partial_t  \int
     D^{s}  \tilde{S}  \Im \langle{\delta a} ,D^{s}
    \delta a\rangle \dd x \lesssim% 
    \| (\delta
     a, \delta u) \|_{X^s}^2 + \| (\delta \tilde{a}, \delta \tilde{u})
     \|_{Y}^2 . \end{equation*}
  Integration from $0$ to $t$ yields
  \begin{align}
    & \varepsilon^2 \| D^{s} \delta a\|^2_{L^2} - 2 \varepsilon
     \int D^{s}  \tilde{S}  \Im \langle{\delta a} ,D^{s}
    \delta a\rangle \dd x %\nonumber\\
    \lesssim% & 
    \int_0^t \| (\delta a, \delta u) (r)
    \|_{X^s}^2 \dd r + t \| (\delta \tilde{a}, \delta \tilde{u}) \|_{Y}^2 .  \label{delta int est}
  \end{align}
  The second term of the LHS is bounded by
  \begin{align*}
    \left| 2 \varepsilon  \int D^{s} 
    \tilde{S}  \Im \langle{\delta a} ,D^{s}
    \delta a\rangle \dd x
    \right| & \leq  C_{L,s} \varepsilon \| (\delta a, \delta u) (r) \|_{X^s} \| D^{s} \delta a\|_{L^2}\\
    & \leq  \frac{C_{L,s}}{\mu_1} (\| (\delta a, \delta u) \|^2_{X^s}
    + \mu_1^2 \varepsilon^2 \| D^{s} \delta a\|^2_{L^2}),
  \end{align*}
Choose $\mu_1$ small enough such
  that
  \begin{equation*} \mu_1 C_{L,s} %| \{ \alpha : |\alpha| = s \} |
  \leq \frac{1}{2},
  \end{equation*}
  then,
  \begin{equation*}   \left| 2 \varepsilon 
     \int D^{s}  \tilde{S}\Im \langle{\delta a} ,D^{s}
    \delta a\rangle \dd x \right| \lesssim \frac{1}{2\mu_1^2} (\|
     (\delta a, \delta u) \|^2_{X^s} + \mu_1^2 \varepsilon^2 \|
     D^{s} \delta a\|^2_{L^2}) . \end{equation*}
  It follows from \eqref{delta int est} that
  \begin{align*}
    & \varepsilon^2 \| \delta a\|^2_{{H}^s} - \frac{1}{2 \mu_1^2} (\|
    (\delta a, \delta u) (r) \|^2_{X^s} + \mu_1^2 \varepsilon^2 \|
    D^{s} \delta a\|^2_{L^2})\\
    &\qquad \qquad \lesssim  \varepsilon^2 \|
    \delta a\|^2_{H^s} - 2 \varepsilon \sum_{|\alpha| = s}  
    \int \partial^{\alpha}_x  \tilde{S} \Im \langle{\delta a} ,\partial^{\alpha}_x
    \delta a\rangle \dd x\\
    &\qquad \qquad \lesssim  \int_0^t \| (\delta a, \delta u) (r)
    \|_{X^s}^2 \dd r + t \| (\delta \tilde{a}, \delta \tilde{u}) \|_{Y}^2 .
  \end{align*}
  Combine this estimate and \eqref{delta Xs est} and it follows that
  \begin{equation}
    \| (\delta a, \delta u) \|_{X^s}^2 + \mu_1^2 \varepsilon^2 \|
    D^{s} \delta a\|^2_{L^2} \lesssim \int_0^t
    \| (\delta a, \delta u) (r) \|_{X^s}^2 \dd r + t \| (\delta
    \tilde{a}, \delta \tilde{u}) \|^2_{Y} .
    \label{unique step 1}
  \end{equation}
  
  \emph{Step 3.}
  For $\partial_t \delta u$, we have the following estimate
  \begin{equation*} \| \partial_t \delta u \|_{H^{s - 1}} \leq C _{\mu_1, L} \left( \|
     (\delta a, \delta u) \|_{X^s} + \| (\delta \tilde{a}, \delta \tilde{u})
     \|_{Y} \right) \end{equation*}
  Choose $\mu_2$ small enough such
  that
  \begin{equation*} \mu_2 C _{\mu_1, L} \leq \frac{1}{4} . \end{equation*}
  Then
  \begin{equation*} \mu_2^2 \| \partial_t \delta u \|_{H^{s - 1}}^2 \leq \frac{1}{4}
     \left( \| (\delta a, \delta u) \|_{X^s}^2 + \| (\delta \tilde{a}, \delta
     \tilde{u}) \|_{Y}^2 \right) . \end{equation*}
     
     \emph{Step 4.}
  Combine this estimate with \eqref{unique step 1} and get
  \begin{align*}
    & \| (\delta a, \delta u) \|_{X^s}^2 + \mu_1^2 \varepsilon^2 \|
     \delta a\|^2_{{H}^{s}} + \mu_2^2 \| \partial_t \delta u
    \|_{H^{s - 1}}^2\\
    \leq & \frac{1}{2} \| (\delta \tilde{a}, \delta \tilde{u})
    \|_{Y}^2 + C \left( \int_0^t \|
    (\delta a, \delta u) (r) \|_{X^s}^2 \dd r + t \| (\delta \tilde{a},
    \delta \tilde{u}) \|^2_{Y} \right) .
  \end{align*}
  By Gronwall's inequality,
  \begin{equation*} \| (\delta a, \delta u) \|_{Y}^2 \leq  \left(
     \frac{1}{2} + C T \right) e^{C T}
     \| (\delta \tilde{a}, \delta \tilde{u}) \|_{Y}^2 .
  \end{equation*}
  For $s=0$, the term $\|\nabla u^1\|_{H^s}$ is controlled by the a priori estimate. We conclude that the sequence is Cauchy in $Y^{T,0}_{\mu_1,\mu_2}$. By the a priori estimate it is also bounded in $Y^{T,s}_{\mu_1,\mu_2}$ and by interpolation we obtain that it is Cauchy for all $Y^{T,s'}_{\mu_1,\mu_2}$ with $s' < s$.
\end{proof}

\section{Behavior at Blowup Time}\label{section blow up}
We now prove Theorem \ref{thm:main result blow up}. We only prove the case of the Pauli-Poisswell-WKB equation \eqref{Madelung-Poisson}-\eqref{wkb data}, since the other part is a direct consequence. %It is equivalent to Proposition \ref{prop: blow up}. %Again, by a similar but simpler argument, we easily obtain the first part in Theorem \ref{thm:main result blow up}, i.e. the blow up of the Euler-Poisswell equation \eqref{Euler-Poisson}.
\begin{proposition}
  \label{prop: blow up}
  Under the assumptions of Proposition \ref{existence} there
  exists a unique local solution $(a^{\varepsilon}, u^{\varepsilon})$ to the Pauli-Poisswell-WKB equation \eqref{Madelung-Poisson}-\eqref{A in wkb}. Suppose the solution is defined on a maximal
  time interval $[0,T^{\ast})$. Then we have
  \begin{equation*}
    \lim_{t \rightarrow T^{\ast} -} M(t) =
    + \infty . \label{blow up sign}
  \end{equation*}
  where $M$ is defined by \eqref{M def}. Furthermore, when $\varepsilon$ is small enough, there exists a $K = K_{\varepsilon, Q, T^{\ast}}$, such that
  \begin{equation*} \limsup_{t \rightarrow T^{\ast} -} \|u^{\varepsilon} (t)
     \|_{W^{1, \infty}} + \|a^{\varepsilon} (t) \|_{L^{\infty}} > K. \end{equation*}
  Here, for fixed $Q$ and $T^{\ast}$, $K_{\varepsilon, Q, T^{\ast}}$ goes to
  infinity as $\varepsilon$ goes to zero.
\end{proposition}

\begin{proof}
  Suppose that $T^{\ast}$ is the maximal time of existence of the solution. Then $E_s^1(t)$ blows up as $t \uparrow T^{\ast}$. If $N (t)$ would remain bounded when $t \uparrow T^{\ast}$, then by Proposition \ref{a priori opt} and
  the assumption of Theorem \ref{thm:main result wellposedness}, $E_s^1(t)$ would remain bounded when $t \uparrow T^{\ast}$ which is a contradiction. Thus,
  \begin{equation*}
      \lim_{t \uparrow T^{\ast}} N (t)
     = + \infty . 
  \end{equation*} 
  On the other hand, $N
  (t)$ remains bounded when $t < T^{\ast}$. It follows that
  \begin{equation*} \lim_{t \uparrow T^{\ast}} M (t)
     = + \infty . \end{equation*}
  For the second part we are going to use a
  bootstrap argument. If for all times
  \begin{equation}
    \|u^{\varepsilon} (t) \|_{W^{1, \infty}} + \|a^{\varepsilon} (\cdot, t) \|_{L^{\infty}} \leq K, \label{contradict cond}
  \end{equation}
  for some constant $K$ to be assigned we also assume that
  \begin{equation}
    \|a^{\varepsilon} (t) \|_{H^1} + \|a^{\varepsilon} (t) \|_{W^{1, \infty}}% + \|a^{\varepsilon} (t) \|_{W^{2, 3}} 
    \leq \frac{K}{ \varepsilon}, \label{first est}
  \end{equation}
  for all times. Here we also require that
  \begin{equation}
    \|a^{\varepsilon,0} \|_{H^1} + \|a^{\varepsilon,0} \|_{W^{1, \infty}} %+ \|a^{\varepsilon,0} \|_{W^{2, 3}} 
    \leq
    \frac{K}{ \varepsilon} . \label{id boot cond}
  \end{equation}
  It follows that for all $t \in [0, T^{\ast})$,
  \begin{equation*}  N(t)
     \leq 2 K + 1. \end{equation*}
  Since $s > 7/2$, by Sobolev's inequality,
  \begin{equation*}
    \|a^{\varepsilon} (t) \|_{H^1} + \|a^{\varepsilon} (t) \|_{W^{1, \infty}} %+ \|a^{\varepsilon} (t) \|_{W^{2, 3}} 
    \lesssim \|a\|_{H^s}.
  \end{equation*} 
  Due to the a priori estimate \eqref{energy a priori},
  \begin{align}
    \|a^{\varepsilon} (t) \|_{H^1} + \|a^{\varepsilon} (t) \|_{W^{1, \infty}} %+ \|a^{\varepsilon} (t) \|_{W^{2, 3}} 
    &\lesssim  E_s^1 (t) \nonumber \\
    &\lesssim Q N (t)^{2 s + 3}  e^{C N
     (t)^{2 s + 3} t}\nonumber \\
    &\lesssim Q (2 K + 1)^{2 s + 3}  e^{C (2 K + 1)^{2 s + 3} T^{\ast}} .\label{second est}
  \end{align}
   If we have
  \begin{equation}
    CQ (2 K + 1)^{2 s + 3}  e^{C (2 K + 1)^{2 s + 3} T^{\ast}} <
    \frac{K}{\varepsilon}, \label{boot cond}
  \end{equation}
  then \eqref{second est} is a stronger estimate than \eqref{first est}. By
  a bootstrap argument, \eqref{second est} holds
  for all times. Then \eqref{contradict cond} implies that $N (t)$ is bounded for all times, which
  contradicts the fact that $T^{\ast}$ is  maximal. Thus we only need to show that when $\varepsilon$ is small enough, there
  exists $K = K (\varepsilon, Q, T^{\ast})$ such that \eqref{boot cond} holds
  true, where we also require that $K > C Q$. Take
 \begin{equation*} K = \frac{1}{2}\left( \left| \log \frac{\sqrt{\varepsilon}}{C T^{\ast}}
     \right|^{1/(2 s + 3)} - 1\right),
     \end{equation*}
     which goes to infinity as $\varepsilon
\rightarrow 0$.
 Then \eqref{boot cond} becomes 
 \begin{equation*}
  C Q \sqrt{\varepsilon} \log
    \left|\frac{\sqrt{\varepsilon}}{C T^{\ast}}\right|  < \frac{1}{2\varepsilon}\left(  \left|\log \frac{\sqrt{\varepsilon}}{C T^{\ast}}\right|
     ^{1/(2 s + 3)} - 1\right),
\end{equation*} 
which holds for $\varepsilon$ is small enough. This finishes the proof.
\end{proof}

\section{Semiclassical limit}\label{section semiclassical limit}

\begin{proof}[Proof of Theorem \ref{thm:main result semiclassical limit}]
Again, we only concentrate on the Poisswell equations. We split the proof into three parts. In the first part we estimate the difference between $(a^{\varepsilon},u^{\varepsilon})$  and $(a,u)$. In the second part we show the semiclassical limit of the density $\rho^{\varepsilon}$, the current density $J^{\varepsilon}$ and the monokinetic Wigner transform $f^{\varepsilon}$. In the third part we prove the statement about the time of existence. In the following all the constants depend implicitly on $Q,T$.\\
  
For the first part let $(a^{\varepsilon}, u^{\varepsilon})$ be the solution of \eqref{Madelung-Poisson}-\eqref{wkb data} and $(a, u)$  the solution of  \eqref{Euler-Poisson}-\eqref{Euler Poisswell limit data}. Set
  \begin{align} \delta a &:= a^{\varepsilon} - a, & \delta u &:= u^{\varepsilon} -
     u, & \delta A &:= A^{\varepsilon} - A, \\ 
     \delta V &:= V^{\varepsilon}
     - V, & \delta {B} &:= {B}^{\varepsilon} - {B}, & \delta
     \rho &:= \rho^{\varepsilon} - \rho .
     \end{align}
   Subtracting the equations yields 
  \begin{align}
  \label{sub system}
      \partial_t \delta a + (\delta u - \delta A) \cdot \nabla
      a^{\varepsilon} + (u-A) \cdot \nabla \delta a & \nonumber\\
      + \frac{1}{2} a^{\varepsilon} \mathrm{div} (\delta u - \delta A) +
      \frac{1}{2} \delta a \mathrm{div} (u-A) & = \frac{i}{2} (\sigma \cdot \delta
      {B})a^{\varepsilon} + \frac{i}{2} (\sigma \cdot {B}) \delta a
      + \frac{i \varepsilon}{2} \Delta a^{\varepsilon},\\
      \partial_t \delta u + (\delta u - \delta A) \cdot \nabla u^{\varepsilon}
      + (u-A) \cdot \nabla \delta u + \delta u \cdot \nabla
      A^{\varepsilon} & \nonumber \\
      + u^{\varepsilon} \cdot \nabla \delta A + \nabla (\frac{1}{2} \delta A
      \cdot (A + A^{\varepsilon}) + \delta V) & = 0,\\
       - \Delta \delta V & = \delta \rho,  %\label{V in sub}
    \\
    - \Delta \delta A & = \varepsilon w^{\varepsilon}
    + \varepsilon v^{\varepsilon}  + (\delta \rho) u^{\varepsilon} \nonumber \\ &\qquad - \rho \delta
    u - (\delta \rho) A + \rho^{\varepsilon} \delta A.  \label{A in sub ori}\\
      (\delta a,\delta u)  (x,0) & = (a^{\varepsilon, 0} - a^0,u^{\varepsilon, 0} - u^0).
  \end{align}
  Apply $D^{s-3}$ to \eqref{sub system},
  \begin{equation*}
    \partial_t D^{s-3} \delta a + (u-A) \cdot \nabla
    D^{s-3} \delta a + \mathbf{R}_f = 0, \label{delta a eq}
  \end{equation*}
  with
  \begin{align*}
    \mathbf{R}_f &:= D^{s-3}  ((\delta u - \delta A) \cdot \nabla
    a^{\varepsilon}) + D^{s-3}  ((u-A) \cdot \nabla \delta a)
    - (u-A) \cdot \nabla D^{s-3} \delta a\\
    &+ \frac{1}{2} D^{s-3} (a^{\varepsilon} \mathrm{div} (\delta
    u - \delta A)) + \frac{1}{2} D^{s-3}  (\delta a \mathrm{div} (u - A))\\
    &- \frac{i}{2} D^{s-3} (\sigma \cdot \delta {B})
    a^{\varepsilon}) + \frac{i}{2} D^{s-3}  ((\sigma \cdot {B} )\delta
    a) - \frac{i \varepsilon}{2} \Delta D^{s-3} a^{\varepsilon}
    . \\
    &= F_1 + F_2 + F_3+ F_4+ F_5+ F_6+ F_7+ F_8
  \end{align*}
  Due to Lemma \ref{charge cons with source}, we
  obtain
  \begin{equation} \partial_t \| D^{s-3} \delta a \|_{L^2}^2 = \int
     \mathrm{div} (u-A)  | D^{s-3} \delta a |^2 \dd x - \int
     \Re \langle D^{s-3} {\delta a},\mathbf{R}_f  \rangle. \end{equation}
     The first term on the RHS can be estimated by $\| \delta a\|_{H^{s - 3}}^2$. Next, we investigate the second term:
     
     \emph{Step 1: Estimate for $\mathbf{R}_f$.}
  Together with the Kato-Ponce inequality we obtain the following
  estimates. For $s-3>0$,
  \begin{align*}
    \|F_1\|_{L^2} %= \| \partial^{\alpha} ((\delta A + \delta u) \cdot \nabla a^{\varepsilon})\|_{L^2} 
    &\lesssim \| \delta u - \delta A\|_{\dot{H}^{s-3}} \| \nabla a^{\varepsilon} \|_{L^{\infty}} +
    \|a^{\varepsilon} \|_{H^{s - 2}}  \| \delta u -\delta A\|_{L^{\infty}}\\
    &\lesssim \| \delta u - \delta A\|_{\dot{H}^{s-3}} + \| \delta u
    - \delta A\|_{L^{\infty}}\\
    &\lesssim \| \nabla \delta A\|_{H^{s - 3}} + \| \delta u\|_{H^s},
  \end{align*}
  and
  \begin{align*}
     \|F_2+F_3\|_{L^2} %&= \| \partial^{\alpha} ((u-A) \cdot \nabla \delta a) - (u-A) \cdot \nabla \partial^{\alpha} \delta a \|_{L^2} \\
    &\lesssim \|u-A\|_{\dot{H}^{s-3}} \| \nabla \delta a \|_{L^{\infty}} + \|
    \delta a \|_{H^{s - 3}}  \|u-A\|_{W^{1, \infty}}\\
    &\lesssim \| \delta a \|_{H^{s - 3}} .
  \end{align*}
  In the $L^2$ case, i.e. for $s-3 = 0$,
  \begin{align*}
     \|F_1\|_{L^2} %= \| \partial^{\alpha} ((\delta u - \delta A) \cdot \nabla a^{\varepsilon})\|_{L^2} 
     & \lesssim \|a^{\varepsilon} \|_{H^1}  \|
    \delta u - \delta A\|_{L^{\infty}}  \lesssim \| \nabla \delta A\|_{H^{s - 3}} + \| \delta u\|_{H^s},
  \end{align*}
  and $F_2+F_3=0$.
  %\begin{align*}
       %\| \partial^{\alpha} ((u-A) \cdot \nabla \delta a) - (u-A) \cdot \nabla \partial^{\alpha} \delta a \|_{L^2} 
  %   = 0.
  %\end{align*}
  For $F_4$ we have,
  \begin{align*}
     \|F_4\|_{L^2} \lesssim \| D^{s-3} (a^{\varepsilon} \mathrm{div} (\delta u - \delta
    A))\|_{L^2} & \lesssim \| \delta u - \delta A\|_{\dot{H}^{|\alpha|}} \|
    \nabla a^{\varepsilon} \|_{L^{\infty}} + \|a^{\varepsilon} \|_{H^{s -
    1}}  \| \delta A + \delta u\|_{W^{1, \infty}}\\
    & \lesssim \| \nabla \delta A\|_{H^{s - 3}} + \| \delta u\|_{H^s}
  \end{align*}
  For $F_5$ we have
  \begin{align*}
     \|F_5\|_{L^2} \lesssim \| D^{s-3} (\delta a \mathrm{div} (u-A))\|_{L^2} & \lesssim
    \|u-A\|_{\dot{H}^{|\alpha|}} \| \nabla \delta a \|_{L^{\infty}} + \|
    \delta a \|_{H^{s - 3}}  \|u-A\|_{W^{1, \infty}}\\
    & \lesssim \| \delta a \|_{H^{s - 3}}.
  \end{align*}
  For $F_6$ we have
  \begin{align*}
    \|F_6\|_{L^2} %\lesssim \| \partial^{\alpha} (\delta \mathbf{B} a^{\varepsilon})\|_{L^2} &\lesssim \| \delta \mathbf{B} \|_{H^{s - 3}} \|a^{\varepsilon}\|_{L^{\infty}} + \|a^{\varepsilon} \|_{H^{s - 3}} \| \delta \mathbf{B}\|_{W^{1, \infty}}\\
    & \lesssim \| \nabla \delta A\|_{H^{s - 3}} \|a^{\varepsilon}
    \|_{L^{\infty}} + \|a^{\varepsilon} \|_{H^{s - 3}}  \| \nabla \delta
    A\|_{W^{1, \infty}}\\
    & \lesssim \| \nabla \delta A\|_{H^{s - 3}}.
  \end{align*}
  For $F_7$ we have
  \begin{align*}
    \|F_7\|_{L^2} \lesssim %\| \partial^{\alpha} (\mathbf{B}\delta a)\|_{L^2} 
    & \lesssim \|
    \nabla A \|_{H^{s - 3}} \| \delta a\|_{L^{\infty}} + \| \delta a\|_{H^{s
    - 3}} \| \nabla A \|_{L^{ \infty}}\\
    & \lesssim \| \delta a\|_{H^{s - 3}},
  \end{align*}
  And for $F_8$, by Proposition \ref{existence}, we have
  \begin{equation*} \|F_8\|_{L^2} \lesssim \varepsilon \| \Delta D^{s-3} a^{\varepsilon} \|_{L^2} \lesssim
     \varepsilon \|a^{\varepsilon} \|_{H^{s - 1}} \\
     \lesssim \varepsilon. 
    \end{equation*}
  Similarly to \eqref{Rc est},
  \begin{equation*} \|\mathbf{R}_f \|_{L^2} \lesssim\| \nabla \delta A\|_{H^{s - 3}} + \| \delta
     u\|_{H^{s - 2}} + \| \delta a\|_{H^{s - 3}} + \varepsilon . \end{equation*}
     
     \emph{Step 2: Estimates for $\|\delta u\|_{H^{s-2}}$ and $\|\delta a\|_{H^{s-3}}$}.
  Similar to \eqref{a a priori estimate}, it follows that
  \begin{equation}
    \partial_t \| \delta a\|_{H^{s - 3}} \lesssim\| \delta a\|_{H^{s - 3}}
    + \| \delta u\|_{H^{s - 2}} + \| \nabla \delta A\|_{H^{s - 3}} +
    \varepsilon . \label{delta a est}
  \end{equation}
  Similarly for $\delta u$,
  \begin{equation}
    \partial_t  \| \delta u\|_{H^{s - 2}} \lesssim \| \delta u\|_{H^{s - 2}}
    + \| \nabla \delta A\|_{H^{s - 3}} + \| \nabla \delta V\|_{H^{s - 1}} .
    \label{delta u est}
  \end{equation}
  
\emph{Step 3: Estimates for $\delta V$ and $\delta A$.}
  For $\delta V$, similarly to \eqref{nabla V Hs bound}, we have
  \begin{equation}
    \| \nabla \delta V\|_{H^{s - 1}} \lesssim \| \delta a\|_{H^{s - 3}} .
    \label{delta V est}
  \end{equation}
  For $\delta A$, similarly to \eqref{VAL2} and \eqref{VLA3},  we have
  \begin{align*}
    \| \nabla \delta A\|_{L^2} & \lesssim \| \delta u\|_{L^2} + \| \delta
    a\|_{L^2} + \varepsilon,\\
    \| \nabla \delta A\|_{H^{s - 3}} & \lesssim \varepsilon + \| \delta
    u\|_{H^{s - 3}} + \| \delta a\|_{H^{s - 3}} + \| \nabla \delta A\|_{L^2} .
  \end{align*}
  It follows that
  \begin{equation}
    \| \nabla \delta A\|_{H^{s - 3}} \lesssim\varepsilon + \| \delta
    u\|_{H^{s - 3}} + \| \delta a\|_{H^{s - 3}} . \label{delta A est}
  \end{equation}
  Combining \eqref{delta a est}, \eqref{delta u est}, \eqref{delta V est} and
  \eqref{delta A est} yields
  \begin{equation*} \partial_t  (\| \delta a\|_{H^{s - 3}} +\| \delta u\|_{H^{s - 2}}) \lesssim \| \delta a\|_{H^{s - 3}} + \| \delta u\|_{H^{s - 2}} + \varepsilon . \end{equation*}
  By Gronwall's inequality
  \begin{equation*} \| \delta a\|_{H^{s - 3}} + \| \delta u\|_{H^{s - 2}} \lesssim\| a^{\varepsilon, 0} - a^0\|_{H^{s - 3}} + \| u^{\varepsilon, 0} - u^0 \|_{H^{s - 2}} + \varepsilon . \end{equation*}
 Then the semiclassical limit \eqref{semiclassical limit eq} immediately follows from the assumptions of Theorem \ref{thm:main result wellposedness}. This finishes the proof
  of the first part of Theorem \ref{thm:main result semiclassical limit}.
  
  Now we turn to the second part of Theorem \ref{thm:main result semiclassical limit},
  i.e. the semiclassical limit of the density $\rho^{\varepsilon}$, the
  current $J^{\varepsilon}$ and the Wigner transform $f^{\varepsilon}$. Recall
  that
  \begin{align*}
    \rho^{\varepsilon} & =  |a^{\varepsilon} |^2,\\
    J^{\varepsilon} & =  \varepsilon w^{\varepsilon} + \varepsilon
    v^{\varepsilon} +\rho^{\varepsilon}  (u^{\varepsilon} - A^{\varepsilon})
    .
  \end{align*}
  By the uniform boundedness of $(a^{\varepsilon},u^{\varepsilon})$ in $X^s$ we can pass to the limit $\varepsilon
  \rightarrow 0$,
  \begin{align*}
    \rho^{\varepsilon} & \underset{\varepsilon \rightarrow 0}{\longrightarrow}
     \rho = |a|^2,\\
     J^{\varepsilon} %= \varepsilon w^{\varepsilon} + \varepsilon v^{\varepsilon} + \rho^{\varepsilon}  (u^{\varepsilon} - A^{\varepsilon}) &
    &\underset{\varepsilon \rightarrow 0}{\longrightarrow}  J = \rho (u -
    A),%\\
    %\varepsilon w^{\varepsilon} + \varepsilon v^{\varepsilon} &\underset{\varepsilon \rightarrow 0}{\longrightarrow} & 0.
  \end{align*}
  in $H^{s - 3}$.  Moreover, there exists a subsequence
  $f^{\varepsilon}$ and a non-negative Radon measure $f$ such that
  \begin{equation*} f^{\varepsilon} \underset{\varepsilon \rightarrow 0}{\rightharpoonup} f
     \quad \text{in } \mathcal{A}', \end{equation*}
     cf. {\cite{GMMP}}. It is easy to see that
  \begin{equation*} \| \varepsilon \nabla \psi^{\varepsilon} \|_{L^2} \leq \|(\varepsilon \nabla -iA^{\varepsilon})\psi^{\varepsilon}\|_{L^2} + \|A^{\varepsilon}\psi^{\varepsilon}\|_{L^2}\leq C, \end{equation*}
  for some constant $C$ independent of $\varepsilon$. Following
  {\cite{LionsPaul}}, it holds that
  \begin{equation*} \int_{\mathbb{R}^3} f (x, \xi, t) \dd \xi = \lim_{\varepsilon \rightarrow
     0} \rho^{\varepsilon} (x,t) = \rho (x,t) . \end{equation*}
  It is easy to see that
  \begin{equation*} \|(u (x, t) - i \varepsilon\nabla \psi^{\varepsilon} \|_{L^2}
     \underset{\varepsilon \rightarrow 0}{\longrightarrow} 0. \end{equation*}
  Following {\cite{2003Wigner}}, it holds that
  \begin{equation*} \int |\xi - (u (x,t)|^2 f (x, \xi, t) \dd \xi \dd x = 0, \end{equation*}
  and
  \begin{equation*} f (x, \xi, t) = \rho (x,t) \delta (\xi - u (x, t)) . \end{equation*}
  It is easy to see that $f$ satisfies \eqref{eq:vlasov_limit} and
  \eqref{eq:vlasov_initial_data} with the change of variables $p = \xi-A(x,t)$.
  
  The last part of Theorem \ref{thm:main result semiclassical limit} is proved by a bootstrap argument. For
  any $T < T^0$, there exists $K = K (T) > 0$, such that for all $t\in [0,T)$,
  \begin{equation*} \| (a, u) (t) \|_{X^s} + N (a, u) (t) \leq K. \end{equation*}
  Suppose that
  \begin{equation} \label{eq:3K est} E_s^1 (a^{\varepsilon}, u^{\varepsilon}) (t) \leq
     K', \end{equation}
   for all $t\in [0,T)$ with $K'$ to be determined. By virtue of the first part of Theorem \ref{thm:main result semiclassical limit}, it follows that for all $t\in [0,T)$,
  \begin{equation*}  \| (a, u) (t) - (a^{\varepsilon}, u^{\varepsilon}) (t)
     \|_{X^{s - 2}} \lesssim_{K, K', T} \varepsilon . \end{equation*}
  Using Sobolev' inequality we obtain
  \begin{equation*} \|u - u^{\varepsilon} \|_{W^{1, \infty}} + \|a - a^{\varepsilon} \|_{H^1}
     + \|a - a^{\varepsilon} \|_{W^{1, \infty}} %+ \|a - a^{\varepsilon}\|_{W^{2, 3}} 
     \lesssim_{K, Q, T} \varepsilon . \end{equation*}
  When $\varepsilon$ is small enough, this estimate yields for all $t\in [0,T)$,
  \begin{equation*}  \| N (a^{\varepsilon}, u^{\varepsilon}) (t) \|_{X^s}
     \leq 2 K, \end{equation*}
  By Proposition \ref{a priori opt} we
  know that for all $t\in [0,T)$,
  \begin{equation*}
     E_s^1 (a^{\varepsilon}, u^{\varepsilon}) (t) \leq CQ (2
    K)^{2 s + 3} e^{C (2 K)^{2 s + 3} T} .
  \end{equation*}
  If we take $K'$ larger than $CQ (2 K)^{2 s + 3} e^{C (2 K)^{2 s + 3} T}$
  then the above estimate is sharper than \eqref{eq:3K est}. By a bootstrap
  argument,  \eqref{eq:3K est} holds and $(a^{\varepsilon},
  u^{\varepsilon})$ does not blow up at $T$. It follows that for all $T< T^0$,
  \begin{equation*} \liminf_{t \rightarrow T^{\varepsilon} -}
     T^{\varepsilon} \geq T. \end{equation*}
     This completes the proof.
\end{proof}

\appendix

\section{Wigner transform and Wigner equation}
\label{sec:Wigner}

In the phase-space formulation of quantum mechanics the {Wigner matrix} $F^{\varepsilon}$ \cite{GMMP} of the 2-spinor $\psi^{\varepsilon}$ is defined as
\begin{equation}
    F^{\varepsilon}(x,\xi,t) := \frac{1}{(2\pi)^{3}} \int_{\mathbb{R}^3} e^{-i\xi \cdot y} \psi^{\varepsilon}(x+\frac{\varepsilon y}{2},t) \overline{\psi^{\varepsilon}(x-\frac{\varepsilon y}{2},t)}^T \dd y,
    \label{eq:wigner_matrix}
\end{equation}
The scalar Wigner transform $f^{\varepsilon}$ of $\psi^{\varepsilon}$ is then given by
\begin{equation}
    f^{\varepsilon} = \Tr F^{\varepsilon}, \label{eq:WT_psi}
\end{equation}
where $\Tr F^{\varepsilon}$ denotes the trace of the $2\times 2$ matrix $F^{\varepsilon}$. The moments of the Wigner transform correspond to the macroscopic densities
\begin{align}
    \rho^{\varepsilon}(x,t) = \int_{\mathbb{R}^3_{\xi}} f^{\varepsilon}(x,\xi,t) \dd \xi,
   % \label{eq:density_wigner}
   &&  J_k^{\varepsilon}(x,t) = \int_{\mathbb{R}^3_{\xi}} \Tr(\sigma \cdot(\xi-A^{\varepsilon}(x))\sigma_k F^{\varepsilon}(x,\xi,t)) \dd \xi.
\label{eq:current_wigner}
\end{align}
However $f^{\varepsilon}$ attains negative values in general which corresponds to the uncertainty
principle for the conjugate variables position and momentum. 

If $\psi^{\varepsilon}$ is a bounded family in $(L^2(\mathbb{R}^3))^2$ then $f^{\varepsilon}$ has a weak limit $f$ in $\mathcal{A}'$ where 
\begin{equation}
    \label{algebra A}
    \mathcal{A} := \{\phi \in C_0(\mathbb{R}^3_x \times \mathbb{R}^3_{\xi}) \colon \mathcal{F}_{\xi}[\phi](x,\eta) \in L^1(\mathbb{R}^3_{\eta}, C_0(\mathbb{R}^3_x ))\}.
\end{equation}
This space is specifically tailored to the Wigner transform and was introduced in \cite{LionsPaul}. 
The limit $f$ is a a non-negative Radon measure, called \emph{Wigner measure} (related to semiclassical measures \cite{gerard1991mesures}) and can be interpreted as a classical phase-space density obeying the Vlasov-Poisswell equation \eqref{eq:vlasov_limit}-\eqref{eq:data_limit} and Vlasov-Darwin equation \eqref{eq:vlasov_darwin_limit}-\eqref{eq:data_limit Darwin}, respectively. 
The scalar Wigner transform $f^{\varepsilon}$ obeys the "quantum Vlasov-Poisswell/Darwin equation"
\begin{align}
       \partial_t f^{\varepsilon} + \xi \cdot \nabla_x f^{\varepsilon} +\tau[A^{\varepsilon}]\nabla_x f^{\varepsilon} +  \xi \theta[A^{\varepsilon}]f^{\varepsilon}
    + \frac{1}{2}\theta[(A^{\varepsilon})^2]f^{\varepsilon}& \nonumber \\ - \frac{\varepsilon}{2}\theta[ B_k^{\varepsilon}]\Tr(\sigma_k F^{\varepsilon}) +\theta[V^{\varepsilon}]f^{\varepsilon}&= 0, \label{eq:pauli_wigner}\\
-\Delta V^{\varepsilon} &= \rho^{\varepsilon}, \\
-\Delta A^{\varepsilon} &= \mathbb{P}J^{\varepsilon} \\
f^{\varepsilon}(x,\xi,0) &= f^{\varepsilon,0}(x,\xi) \\ &= \Tr(F^{\varepsilon,0}(x,\xi)),  \label{eq:wigner matrix initial}
\end{align}
where $\theta[\cdot]$ and $\tau[\cdot]$ are the pseudo-differential operators defined by
\begin{equation}
    (\theta[g]f^{\varepsilon})(x,\xi) := -\frac{1}{(2\pi)^3}\int_{\mathbb{R}^6} \frac{1}{i\varepsilon}(g(x+\frac{\varepsilon y}{2})-g(x-\frac{\varepsilon y}{2}))f^{\varepsilon}(x,\eta) e^{-i (\xi-\eta)\cdot y} \dd \eta \dd y, 
    \label{eq:PDO}
\end{equation}
and
\begin{equation}
    (\tau[g]f^{\varepsilon})(x,\xi)  := -\frac{1}{(2\pi)^3}\int_{\mathbb{R}^6} \frac{1}{2}(g(x+\frac{\varepsilon y}{2})+g(x-\frac{\varepsilon y}{2}))f^{\varepsilon}(x,\eta) e^{-i (\xi-\eta)\cdot y} \dd \eta \dd y. 
\end{equation}
The initial datum is given by
\begin{align}
    F^{\varepsilon,0}(x,\xi) = \frac{1}{(2\pi)^{3}} \int_{\mathbb{R}^3} e^{-i\xi\cdot y} \psi^{\varepsilon,0}(x+\frac{\varepsilon y}{2}) \overline{\psi^{\varepsilon,0} (x-\frac{\varepsilon y}{2})}^T \dd y.
\end{align}
Formally passing to the limit $\varepsilon \rightarrow 0$ in \eqref{eq:pauli_wigner}-\eqref{eq:wigner matrix initial} and taking the trace yields the \textbf{Vlasov-Poisswell equation} \eqref{eq:vlasov_limit}-\eqref{eq:data_limit} and the \textbf{Vlasov-Darwin equation} \eqref{eq:vlasov_darwin_limit}-\eqref{eq:data_limit Darwin}.

Taking $\psi^{\varepsilon,0}$ of the form \eqref{WKB initial} and assuming that $(\rho^{\varepsilon,0},u^{\varepsilon,0})$ converges to $(\rho^0,u^0)$ yields the \emph{monokinetic} Wigner measure
\begin{equation*}
    f^0(x, p) = \rho^{0}(x) \delta(p-u^0(x)).
    %\label{monokinetic Wigner init}
\end{equation*}
If we use this as initial data in \eqref{eq:data_limit} and \eqref{eq:data_limit Darwin} we can expect that for short times, $f(x, p, t)$ is of the form
\begin{equation*}
    f(x, p, t) = \rho(x,t) \delta(p-u(x,t)).
    %\label{monokinetic Wigner}
\end{equation*}

\section*{Acknowledgement}

N.M. and J.M. acknowledge financial support from the Austrian Science Fund (FWF) via the SFB project F65 and the DK project W1245, as well as the Vienna Science and Technology Fund (WWTF) via project MA16-066 "SEQUEX".

J.M acknowledges financial support from the Austrian Science Fund (FWF) via the Schrödinger grant 10.55776/J4840, from Campus France via a "Bourse d'Excellence" scholarship and from Wolfgang Pauli Insitute Vienna for a Pauli scholarship. Also, J.M. thanks Thomas Y. Hou for his hospitality and support at Caltech.

C.Y. is supported by the NSF Grant DMS-2205590. C.Y. acknowledges the hospitality and support of the Wolfgang Pauli Institute Vienna. Further C.Y. expresses his gratitude towards Zhennan Zhou for continuous support.

\bibliographystyle{abbrv}
\bibliography{pauliwkb}

\end{document}